%% file: main.tex
\documentclass[12pt]{amsart}
\usepackage[top=1.5in, left=1.5in, right=1.5in, bottom=1.5in]{geometry}
\usepackage{amsmath}
\usepackage{amsfonts}
\usepackage{amsthm}
\usepackage{amssymb}
\usepackage{graphicx}
\usepackage{tikz}
\usepackage{hyperref}
\usepackage{amsrefs}
\usepackage{msc}

\title[Topology of hyperbolic $\mathbb{R}^n$-actions on $n$-manifolds]{Some remarks on the topology of hyperbolic actions of $\mathbb{R}^n$ on $n$-manifolds}
\author{Damien Bouloc}
\address{Institut de Math\'ematiques de Toulouse \\ 118, route de Narbonne \\ 31062 Toulouse Cedex 9, France }
\email{damien.bouloc@math.univ-toulouse.fr}
\date{\today}

\newtheorem{theorem}{Theorem}
\newtheorem{corollary}{Corollary}
\newtheorem{proposition}{Proposition}
\newtheorem{lemma}{Lemma}
\theoremstyle{remark}
\newtheorem{remark}{Remark}
\theoremstyle{definition}
\newtheorem{definition}{Definition}

\newcommand*{\Hcal}{\mathcal{H}}
\newcommand*{\Ocal}{\mathcal{O}}
\newcommand*{\Scal}{\mathcal{S}}
\newcommand*{\complexset}{\mathbb{C}}
\newcommand*{\realset}{\mathbb{R}}
\newcommand*{\integerset}{\mathbb{Z}}
\newcommand*{\modintegerset}[1]{\mathbb{Z}/{#1}\mathbb{Z}}
\newcommand*{\modclass}[1]{\overline{#1}}
\newcommand*{\projspace}[1]{\realset\mathbb{P}^{#1}}
\newcommand*{\slope}[2]{[#1 : #2]}
\newcommand*{\sphere}[1]{S^{#1}}
\newcommand*{\closeddisk}{\overline{D^2}}
\newcommand*{\torus}[1]{\mathbb{T}^{#1}}
\newcommand*{\lensspace}[2]{L(#1;#2)}
\newcommand*{\set}[1]{\left\lbrace #1 \right\rbrace}
\newcommand*{\closure}[1]{\overline{#1}}
\newcommand*{\boundary}[1]{\partial #1}
\newcommand*{\abs}[1]{\left| #1 \right|}
\newcommand*{\partdiff}[2]{\frac{\partial #1}{\partial #2}}
\newcommand*{\pair}[2]{\set{#1, #2}}
\newcommand*{\norm}[1]{\| #1 \|}
\newcommand*{\Norm}[1]{\left \| #1 \right \|}
\newcommand*{\dd}{\mathrm{d}}
\newcommand*{\e}{\mathrm{e}}

\DeclareMathOperator{\card}{card}
\DeclareMathOperator{\Ind}{Ind}

\begin{document}
	
\begin{abstract}
	This paper contains some more results on the topology of a nondegenerate action of $\realset^n$ on a compact connected $n$-manifold $M$ when the action is totally hyperbolic (i.e. its toric degree is zero). We study the $\realset$-action generated by a fixed vector of $\realset^n$, that provides some results on the number of hyperbolic domains and the number of fixed points of the action. We study with more details the case of the 2-sphere, in particular we investigate some combinatorial properties of the associated 4-valent graph embedded in $\sphere{2}$. We also construct hyperbolic actions in dimension 3, on the sphere $\sphere{3}$ and on the projective space $\projspace{3}$.
\end{abstract}

	
\maketitle

\tableofcontents

\section{Introduction}

In the theory of dynamical systems, integrability appears as a natural condition since it is satisfied by many physical systems, from molecular dynamics to celestial mechanics. Integrable Hamiltonian systems have the particularity to have all locally the same normal form, given by Arnold--Liouville--Mineur theorem, and the latter result has motivated the geometric study of such systems. The research on the subject is still active and of great interest (see e.g. \cite{fomenko2014algebra,marco2013polynomial,vungoc2013}), and a lot of work has be done already on these systems, e.g. on their global topology or geometry, their invariants or their singularities (\cite{babelon2003introduction,bolsinov2004integrable,bolsinov2006singularities,zung1996symplectic} to give a few examples). However, there are also physical systems that are non-Hamiltonian because of the existence of non-holonomic constraints or dissipation phenomenon, but are still integrable in a natural sense. These systems have also been studied from various points of view (see \cite{ayoul2010galoisian,bogoyavlenskij1998extended,bolsinov2014singularities,bountis1984complete,cushman2001non,fedorov2006quasi,ishida2013topological,stolovitch2000singular}) but less is know yet about them compared with the widely developed theory of Hamiltonian systems. Formally, a non-Hamiltonian integrable system is defined by $p$ commuting vector fields and $q$ functions on a manifold of dimension $p+q$. As in the Hamiltonian case, we can impose conditions on the singularities of the system and define the notion of non-degenerate singularity. This way we can restrict our study to systems that are not totally regular, but still reasonable enough to expect some global information on the ambient manifold. Minh and Zung \cite{zung2014geometry} initiated a detailed study of integrable non-Hamiltonian systems with non-degenerate singularities by working on the case $(p,q)=(n,0)$, which is of particular importance because any integrable system is actually of this kind when restricted to adequate submanifolds. In this case the system is nothing more than an action of $\realset^n$ on a $n$-dimensional manifold. Because of isotropy, the action of $\realset^n$ descends to an action of $\torus{k} \times \realset^{n-k}$ on the manifold, and the maximal $k$ satisfying this property is called the toric degree of the action. When it is equal to zero we have what is called a (totally) hyperbolic action. Hyperbolic actions form an interesting subclass of integrable systems, admitting for example a global classification based on the decomposition of the manifold into its $n$-dimensional orbits. The aim of this paper is to provide some more results on totally hyperbolic actions of $\realset^n$ on compact connected $n$-manifolds.

In Section~\ref{s:preliminaries}, we summarize without proofs what is known about totally hyperbolic actions of $\realset^n$ on $n$-manifolds. Further details can be found in the fifth section of \cite{zung2014geometry}. In Section~\ref{s:Morse_etc}, we study the flows generated by restricting the action to a direction in $\realset^n$, and we prove the existence of Morse functions whose singularities are analogous to the singularities of such a flow. In Section~\ref{s:number_of_domains}, we investigate some properties satisfied by the number of hyperbolic domains of a totally hyperbolic action. Section~\ref{s:on_the_2_sphere} is devoted to a more detailed study of hyperbolic actions on the 2-sphere. In this case the decomposition into hyperbolic domains can be seen as the embedding in the sphere of a 4-valent graph, and we investigate some combinatorial conditions it has to satisfy. Finally, Section~\ref{s:in_dimension_3} presents a construction of hyperbolic actions on the 3-dimensional sphere, that extends then naturally to the projective space $\projspace{3}$.

\textbf{Acknowledgments} The author wishes to thank his advisors Philippe Monnier and Nguyen Tien Zung for several helpful comments and many stimulating conversations.

\section{Preliminaries}
\label{s:preliminaries}

In this section, we recall the definition of totally hyperbolic actions and their classification by complete fans described by Zung and Minh in~\cite{zung2014geometry}.

\subsection{Nondegenerate action and toric degree}
Let $\rho : \realset^n \times M \rightarrow M$ be a smooth action of $\realset^n$ on a compact connected $n$-dimensional manifold $M$. It is generated by $n$ commuting vector fields $X_1, \dots, X_n$, defined by
\[ X_i(p) =  \left. \frac{d}{dt} \rho(t e_i, p) \right|_{t=0} \]
for all $p \in M$, where $e_i = (0, \dots, 1, \dots, 0)$ is the $i$-th vector in the canonical basis of $\realset^n$. More generally, if $v = (v^1, \dots, v^n)$ is a vector in $\realset^n$, we define the \emph{the generator of the action $\rho$ associated to $v$} as the vector field
\[ X_v = v^1 X_1 + \cdots + v^n X_n. \]
Recall that the \emph{rank} of a point $p \in M$ is defined as the dimension $r$ of the subspace of $T_p M$ spanned by $X_1(p), \dots, X_n(p)$. The point is said to be \emph{singular} if $r < n$, in particular it is a \emph{fixed point} if $r = 0$.

If $p$ is a fixed point of the action, then the linear parts $X_i^{(1)}$ of the vectors fields $X_i(p)$ at $p$ form a family of commuting well-defined linear vector fields on the tangent space $T_p M$. The linear action induced on the $n$-dimensional vector space $T_p M$ is called the \emph{linear part} of the action $\rho$ at $p$ and is denoted by $\rho^{(1)}$. This linear action is \emph{nondegenerate} if the Abelian Lie algebra spanned by $X_1^{(1)}, \dots, X_n^{(1)}$ is a Cartan sub-algebra of $\mathfrak{gl}(T_p M)$ (that is to say it has dimension $n$ and all of its elements are semi-simple). In this case, we say that the point $p$ is a \emph{nondegenerate} fixed of point of the action.

This definition of non-degeneracy extends to non-fixed singular points as follows. Suppose $p$ has rank $0 < r < n$. Without loss of generality, we may suppose that $X_1(p) = \cdots = X_{k}(p) = 0$ and $X_i = \frac{\partial}{\partial x_i}$ for all $k < i \leq n$ in a local coordinate system around $p$, where $k = n - r$ is the \emph{co-rank} of $p$. The projections of $X_1, \dots, X_k$ induce an infinitesimal action $\tilde{\rho}$ of $\realset^k$ on the local $k$-dimensional manifold $N = \set{x_{k+1} = \cdots = x_n = 0}$. We say that $p$ is \emph{nondegenerate} if its image in $N$ is a nondegenerate fixed point of $\tilde{\rho}$. The action $\rho$ is nondegenerate if all of its singular points are nondegenerate in the above sense.

Denote by 
\[ Z = \set{g \in \realset^n \mid g \cdot p = p \text{ for all } p \in M} \]
the isotropy group of the action $\rho$ on $M$. When $\rho$ is nondegenerate, one can show that it is locally free almost everywhere, and then $Z$ is a discrete subgroup of $\realset^n$. The classification of such subgroups tells us that $Z$ is then isomorphic to $\integerset^t$ for some $0 \leq t \leq n$. The integer $t$ is called the \emph{toric degree} of the action $\rho$. It can be interpreted as the maximal number $t$ such that the action $\rho$ descends to an action of $\torus{t} \times \realset^{n-t}$ on $M$.

We say that the nondegenerate action $\rho$ is \emph{totally hyperbolic} if its toric degree is zero, i.e. if $\rho$ is faithful.

\subsection{Local normal form of totally hyperbolic actions} In a neighborhood of any point of a non-degenerate action $\rho : \realset^n \times M \rightarrow M$, there is a normal form. Let us recall what it looks like when $\rho$ is totally hyperbolic.

Suppose $p \in M$ is a point of rank $r = n - h$ of a totally hyperbolic action $\rho$ of $\realset^n$ on $M$. Then there exists local coordinates $(x_1, \dots, x_n)$ around $p$ and a basis $(v_1, \dots, v_n)$ of $\realset^n$ such that the generators of the action associated to the elements of this basis have the form:
\[ X_{v_1} = x_1 \partdiff{}{x_1}, \dots, X_{v_h} = x_h \partdiff{}{x_h},
\quad
X_{v_{h+1}} = \partdiff{}{x_{h+1}}, \dots, X_{v_n} = \partdiff{}{x_n}, \]
and the $r$-dimensional orbit containing $p$ is locally defined by the equations $\set{x_1 = \dots = x_h = 0}$.
We say that $(x_1, \dots, x_n)$ are \emph{canonical coordinates} and that $(v_1, \dots, v_n)$ is an \emph{adapted basis} of the action $\rho$ around the point $p$.

If $(y_1, \dots, y_n)$ is another system of canonical coordinates with adapted basis $(w_1, \dots, w_n)$, then one can show that the family $(w_1, \dots, w_h)$ is a permutation of $(v_1, \dots, v_h)$. Conversely, if $(w_1, \dots, w_n)$ is a basis of $\realset^n$ such that $(w_1, \dots, w_h)$ is a permutation of $(v_1, \dots, v_h)$, then $(w_1, \dots, w_n)$ is an adapted basis corresponding to some canonical coordinates $(y_1, \dots, y_n)$.

In particular, if $h = 1$, that is on the neighborhood of a point $p$ on a $(n-1)$-dimensional orbit, there is a unique $v_1 \in \realset^n$ such that we have locally $X_{v_1} = x_1 \partdiff{}{x_1}$ in some local coordinates $(x_1, \dots, x_n)$. We say that $v_1$ is the vector associated to the orbit containing $p$.

\subsection{Classification of totally hyperbolic actions}
Suppose $\rho$ is a totally hyperbolic action of $\realset^n$ on a $n$-dimensional compact connected manifold $M$. Its $n$-dimensional orbits are called \emph{hyperbolic domains}.

If $\Ocal$ is a hyperbolic domain of $\rho$, then its closure $\closure{\Ocal}$ is a contractible manifold with boundary and corners, and admits a cell decomposition where each $k$-dimensional cell is a $k$-dimensional orbit of $\rho$. It is a ``curved polytope'' in the sense that it is very similar, but not necessarily diffeomorphic, to a simple convex polytope in $\realset^n$. For any $p \in \Ocal$ and $w \in \realset^n$, the curve $\rho(-tw, p)$ converges to a point in $\closure{\Ocal}$ when $t$ tends to $+\infty$. By commutativity, the orbit of $\rho$ in which this limit lies does not depend on the point $p \in \Ocal$. Thus we can decompose $\realset^n$ into sets
\[ C_{\Hcal} = \set{w \in \realset^n \mid \lim\limits_{t \to \infty} \rho(-tw, p) \in \Hcal \text{ for any } p \in \Ocal} \]
indexed by the orbits $\Hcal \subset \closure{\Ocal}$. Each $\closure{C_\Hcal}$ is a convex cone in $\realset^n$ with simplicial base and dimension $n - \dim \Hcal$. In particular when $\Hcal = \Ocal$ we have $C_\Ocal = \set{0}$, and when $\Hcal = \set{p}$ is a fixed point, $C_{\set{p}}$ is a $n$-dimensional cone. If $\Hcal$ is $(n-1)$-dimensional orbit, then $C_{\Hcal}$ is precisely the one-dimensional cone spanned by the vector $v \in \realset$ associated to the orbit $\Hcal$ that we define above.

\begin{lemma}
	\label{l:gluing_n-1_dim_orbits}
	The $(n-1)$-dimensional orbits of a totally hyperbolic action $\rho : \realset^n \times M \rightarrow M$ can be glued into smooth closed hypersurfaces $H_1, \dots, H_N$ which intersect transversally, and such that two $(n-1)$-orbits lying in a same hypersurface $H_i$ have the same associated vector $v_i \in \realset^n$. 
\end{lemma}

\begin{proof}
	The closures of two $(n-1)$-dimensional orbits are either disjoint, or intersect along the closure of a $(n-2)$-dimensional orbit. Let $p \in M$ be a point of rank $n-2$, with canonical coordinates $(x_1, \dots, x_n)$ and adapted basis $(v_1, \dots, v_n)$. The orbit $\Ocal_p$ containing $p$ is locally defined by $\set{x_1 = x_2 = 0}$. It lies in the closure of exactly four $(n-1)$-dimensional orbits $\Ocal_1^+$, $\Ocal_1^-$, $\Ocal_2^+$, $\Ocal_2^-$, where $\Ocal_i^\pm$ is the orbit locally defined by $\set{x_i = 0, \pm x_j > 0}$ (with $j$ such that $\set{1,2} = \set{i,j}$). It follows that $\Ocal_i^+$ and $\Ocal_i^-$ can be glued along $\Ocal_p$ into a smooth manifold defined locally by the equation $\set{x_i = 0}$. Now it remains to show that the vector associated to $\Ocal_i^\pm$ is precisely $v_i$. We do this for $\Ocal_1^+$, the proof is similar in the other cases. Using a change of coordinates $y_2 = x_2 - \varepsilon$, $y_i = x_i$ for $i \neq 2$, we obtain new local coordinates centered at $q \in \Ocal_1^+$ in which we have $X_{v_2} = (y_2 + \varepsilon) \partdiff{}{y_2}$ while the expressions of the other generators remain unchanged, and then in particular do not depend on $y_2$. It follows that $X_{v_2}$ is rectifiable around $q$ by a change of coordinates that preserves the $y_i$ for $i \neq 2$: we have new coordinates $(z_1, \dots, z_n)$ in which
	\[ X_{v_1} = z_1 \partdiff{}{z_1}, X_{v_2} = \partdiff{}{z_2}, \dots, X_{v_n} = \partdiff{}{z_n}. \]
	It follows that $(z_1, \dots, z_n)$ are canonical coordinates around $q \in \Ocal_1^+$ with the same adapted basis $(v_1, \dots, v_n)$ as for $p$. In particular, that implies that $v_1$ is the vector associated to the orbit $\Ocal_1^+$.
\end{proof}

Let $(H_i)_{1 \leq i \leq N}$ be the family of embedded closed hypersurfaces in $M$ given by the above lemma, with associated vectors $v_1, \dots, v_N$. Fix a hyperbolic domain $\Ocal$ of the action. For each hypersurface $H_i$ there is at most one $(n-1)$-dimensional orbit $\Hcal_i$ in $\closure{\Ocal}$ such that $\Hcal_i \subset H_i$. The family $(C_{\Hcal}, v_i)$ indexed by the orbits $\Hcal \subset \closure{\Ocal}$ and the $i \in \set{1, \dots, N}$ such that $H_i \cap \closure{\Ocal} \neq \emptyset$ defines a \emph{complete fan} of $\realset^n$. If $\Ocal'$ is the hyperbolic domain of another totally hyperbolic action $\rho'$ on $M$ with the same complete fan in $\realset^n$, then there exists a diffeomorphism between $\closure{\Ocal}$ and $\closure{\Ocal'}$ that intertwines the actions $\rho$ and $\rho'$. It follows that hyperbolic actions are classified by their singular hypersurfaces and complete fans.

Conversely, let $M$ be a manifold and $H_1, \dots, H_N$ be embedded closed hypersurfaces that intersect transversely, such that $H_1, \dots, H_N$ split $M$ into compact connected ``curved polytopes'' $\Ocal_1, \dots, \Ocal_F$. Suppose there exists a family $(v_1, \dots, v_N)$ of vectors in $\realset^n$ spanning a complete fan, and such that for each $\Ocal_j$, the subfamily
\[ \set{v_i \mid H_i \cap \closure{\Ocal_j} \neq \emptyset } \]
spans also a complete fan of $\realset^n$ compatible with the combinatorics of the faces of $\closure{\Ocal}$. Then there exists a totally hyperbolic action $\rho$ of $\realset^n$ on $M$ whose hyperbolic domains are exactly $\Ocal_1, \dots, \Ocal_F$ (and thus the $H_i$ are obtained by taking the closure of the $(n-1)$-dimensional orbits of $\rho$).

\section{$\realset$-action generated by a vector}
\label{s:Morse_etc}

Let $M$ be a compact manifold of dimension $n$ with a totally hyperbolic action $\rho : \realset^n \times M \rightarrow M$. Denote by $v_1, \dots, v_N$ the vectors in $\realset^n$ associated to the invariant hypersurfaces $H_1, \dots, H_N$ given by Lemma~\ref{l:gluing_n-1_dim_orbits}.

Fix a generic $w \in \realset^n$ with respect to the family $v_1, \dots, v_N$, in the sense that $w$ does not lie in any vector subspace generated by $v_{i_1}, \dots, v_{i_k}$ if $k < n$. Denote by $\varphi_w^t = \rho(-tw, .)$ the flow of the action in the direction $w$, that is the flow of the generator $-X_w$.

\begin{figure}
	\centering
	\def\svgwidth{0.9\textwidth}
	\small%
	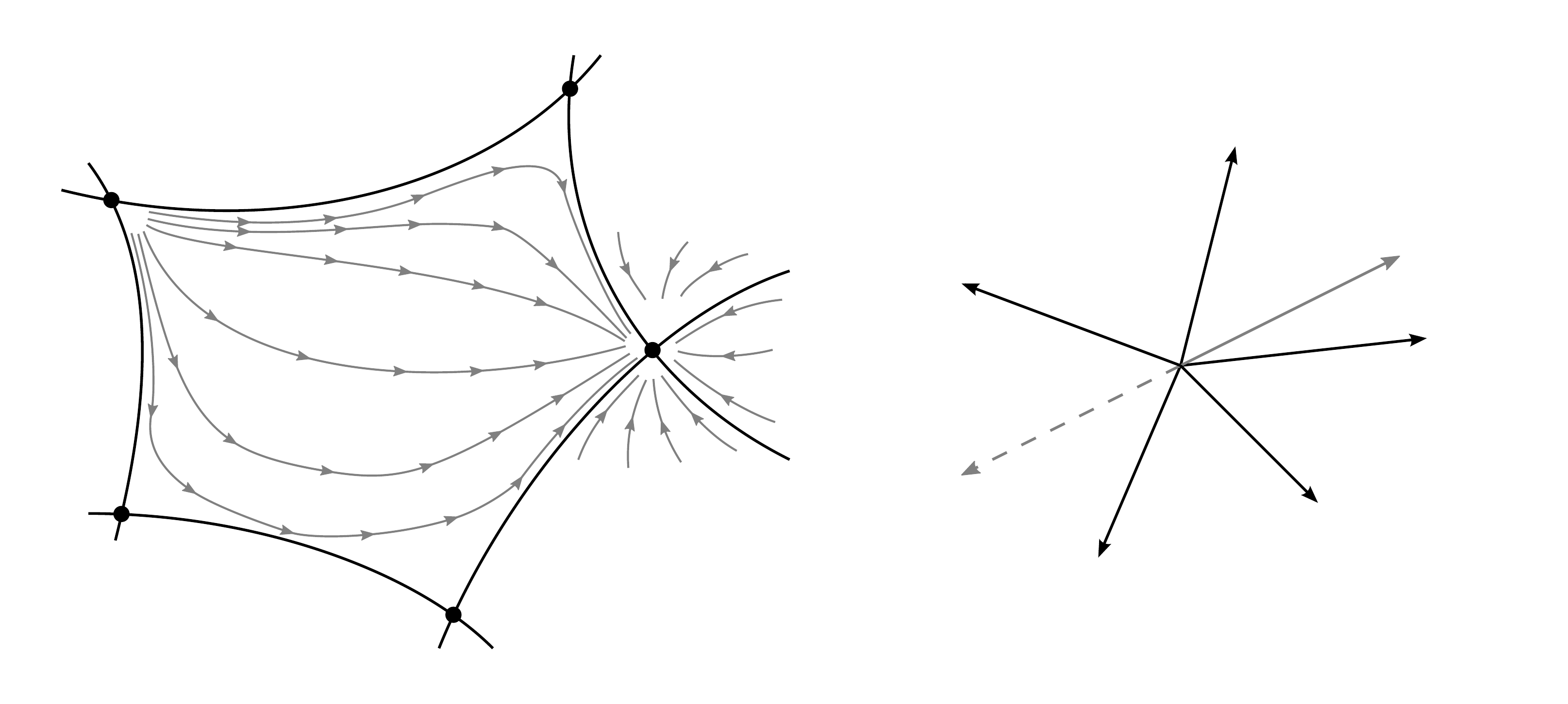
	\caption{$1$-dimensional flow induced by a vector $w \in \realset^n$}
	\label{f:flow_domain}
\end{figure}

Let $\Ocal$ be a hyperbolic domain. Since $w$ is generic, it lies in a $n$-dimensional cone in the decomposition
\[ \realset^n = \bigsqcup_{\Hcal \subset \closure{\Ocal}} C_\Hcal \]
given by the subfamily $v_{i_1}, \dots, v_{i_n}$. By definition of the $C_\Hcal$, for any $p \in \Ocal$ the flow $\varphi_w^t(p)$ tends to a fixed point $p_+ \in \boundary{\Ocal}$ as $t$ tends to $+\infty$. Similarly, $-w$ lies in a $n$-dimensional cone of the complete fan associated to $\Ocal$, and then the flow $\varphi_w^{-t}(p) = \varphi_{-w}^t(p)$ tends to a fixed point $p_- \in \boundary{\Ocal}$ as $t$ tends to $+\infty$ for any $p \in \Ocal$. The fixed points of the flow $\varphi_w^t$ are exactly the fixed points of the action $\rho$. Looking at the dynamics of the flow $\varphi_w^t$ on $\closure{\Ocal}$, we thus observe that the point $p_+$ is attractive, the point $p_-$ is repulsive, and any other fixed point on $\boundary{\Ocal}$ is a saddle point, as illustrated in Figure~\ref{f:flow_domain}.

Remark that the flow $\varphi_w^t$ has the following properties:
\begin{itemize}
	\item there are a finite number of fixed points for $\varphi_w^t$ on $M$,
	\item there are no higher-dimensional closed orbits, and
	\item for any other point $p \in M$,  $\varphi_w^t(p)$ tends to one of the above fixed points as $t$ tends to $+\infty$.
\end{itemize}
However, $\varphi_w^t$ does not satisfy the local structural stability required to be a Morse--Smale flow \cite{smale1960morse}. We will say that $\varphi_w^t$ is the \emph{quasi Morse--Smale flow} of $\rho$ associated to $w$. We can adapt the definition of Morse index to the fixed points of the flow $\varphi_w^t$ as follows.

\begin{definition}
	Let $p$ be a fixed point of the flow $\varphi_w^t$. Let $(v_{i_1}, \dots, v_{i_n})$ be an adapted basis around $p$, with local coordinates $(x_1, \dots, x_n)$. Denote by $(\alpha_1, \dots, \alpha_n) \in \realset^n$ the coordinates of $w$ in this adapted basis. Recall that
	\[ X_w = \alpha_1 x_1 \partdiff{}{x_1} + \cdots + \alpha_n x_n \partdiff{}{x_n}, \]
	or equivalently
	\[ \varphi_w^t(x_1, \dots, x_n) = (x_1 \e^{-\alpha_1 t}, \dots, x_n \e^{-\alpha_n t}). \]
	The \emph{index of $p$ with respect to $w$} is the number
	\[ \Ind_p(w) = \card \set{1 \leq i \leq n \mid \alpha_i > 0} \]
	of attractive directions of $\varphi_w^t$ around $p$.
	
	For instance, $\Ind_p(w)$ is equal to $n$ (res. equal to $0$) if $p$ is an attractive point of $\varphi_w^t$ (res. a repulsive point of $\varphi_w^t$).
\end{definition}

We want to use Morse theory to get information about the number of fixed points of $\rho$ of given index with respect to $w \in \realset^n$. To do so, we would like to construct a Morse function $f : M \rightarrow \realset$ whose singularities are exactly the fixed points of $\rho$, and such that their Morse indices coincide with their indices with respect to $w$. In particular, $f$ has to be increasing in the direction of the flow $\varphi_w^t$, which leads to a first necessary condition on $\varphi_w^t$ for a such function $f$ to exist.

\begin{definition}
	Let $\rho$ be a totally hyperbolic action of $\realset^n$ on a $n$-dimensional manifold $M$, with associated fan $v_1, \dots, v_N$, and let $w \in \realset^n$ be generic with respect to this fan.
	For two fixed points $p_1, p_2$ of $\rho$, write $p_1 \rightarrow p_2$ if there exists a point $p \in M$ (of rank 1) such that
	\[ \lim\limits_{t \to - \infty} \varphi_w^t(p) = p_1 \quad \text{and} \quad \lim\limits_{t \to +\infty} \varphi_w^t(p) = p_2. \]
	We say that the flow $\varphi_w^t$ \emph{has no cycles} if there does not exist a finite sequence $p_1, \dots, p_k$ of fixed points of $\rho$ such that
	\[ p_1 \rightarrow p_2 \rightarrow \cdots \rightarrow p_k \rightarrow p_1. \]
\end{definition}

\begin{lemma}
	\label{l:no_cycles_in_dimension_2}
	If $M$ has dimension $2$, then for any generic $w \in \realset^n$, the flow $\varphi_w^t$ has no cycles.
\end{lemma}

\begin{proof}
	Suppose there exists a cycle
	\[ p_1 \rightarrow p_2 \rightarrow \cdots \rightarrow p_r \rightarrow p_{r+1} = p_1. \]
	Denote by $H_1, \dots, H_r$ the hypersurfaces in $M$ such that, for each $1 \leq i \leq r$, $p_i$ is the intersection between $H_{i-1}$ and $H_i$, and let $v_1, \dots, v_r$ be the corresponding vectors in the complete fan associated to the action (the indices are considered modulo $r$, and we do not exclude that $H_i = H_j$ for some $i \neq j$). Each $H_i$ contains precisely the codimension 1 orbit $O_i$ such that, for any $p \in O_i$, $\rho_w(t,p)$ tends to $p_i$ as $t$ tends to $-\infty$, and to $p_{i+1}$ as $t$ tends to $+\infty$. It follows that each $p_i$ is simultaneously an attractive point on $O_{i-1}$ and a repulsive point on $O_i$, and then according to the local normal form around $p_i$, $w$ satisfies 
	\[ w = -b_i v_i + a_i v_{i+1}, \quad a_i, b_i > 0.  \]
	So, component-wise, the vectors $(v_1, \dots, v_r)$ satisfy the system:
	\[ \left\lbrace	\begin{array}{llllll}
	-b_1 x_1 &+a_1 x_2 &	     &                &            &= w \\
	&-b_2 x_2 &+a_2 x_3 &                &            &= w \\
	&         & \cdots  &                &            &= w \\
	&         &         &-b_{n-1}x_{n-1} &+a_{n-1}x_n &= w \\
	a_n v_1  &         &         &                &-b_n x_n    &= w
	\end{array} \right. \]
	Cramer's rule implies in particular $\Delta v_1 = \delta_1 w$, where $\Delta$ is the determinant of the system and 
	\[ \delta_1 = \sum_{i=0}^{n-1} a_1\dots a_i b_{i+2}\dots b_n > 0. \]
	It follows that $w$ and $v_1$ are linearly dependent, which contradicts the assumption that $w$ is generic.
\end{proof}

We do not know if this result holds in dimension $3$ or more. So in the rest of this subsection, we will always consider the extra assumption that $\varphi_w^t$ has no cycles.
	
\begin{theorem}
	\label{t:construction_of_Morse_functions}
	Let $\rho : \realset^n \times M  \rightarrow M$ be a totally hyperbolic action of $\realset^n$ on a compact connected $n$-manifold $M$. Let $w \in \realset^n$ be generic with respect to the complete fan associated to $\rho$.
	
	If $\varphi_w^t$ has no cycles, then there exists a Morse function $f : M \rightarrow \realset$ such that $p \in M$ is a singular point of $f$ if and only if it is a fixed point of $\rho$, and satisfies
	\[ \Ind_p(f) = \Ind_p(w), \]
	where $\Ind_p(f)$ is the usual Morse index of $f$ at $p$.
\end{theorem}

\begin{proof}
	We construct the function $f : M \rightarrow \realset$ explicitly. Let $p_1, \dots, p_V$ be the fixed points of $f$. Fix $c_1, \dots, c_V \in \realset$ such that
	\[ \forall\ 1 \leq i, j \leq V,\quad p_i \rightarrow p_j \implies c_i < c_j \]
	(it is possible precisely because $\varphi_w^t$ has no cycles). Define $d_i^+ = c_i + \varepsilon$ and $d_i^- = c_i - \varepsilon$ where $0 < \varepsilon < \min_{p_i \rightarrow p_j} (c_j - c_i)$.
	
	\textbf{Step 1.} Take $U_i$ a neighborhood of $p_i$ with local coordinates $(x_1, \dots, x_n)$ centered at $p_i$ associated to an adapted base $B = (v_1, \dots, v_n)$. Let $(\alpha_1, \dots, \alpha_n) \in (\realset \setminus \set{0})^n$ be the coordinates of $w$ in $B$. We have:
	\[ X_w = \alpha_1 x_1 \partdiff{}{x_1} + \cdots + \alpha_n x_n \partdiff{}{x_n}. \]
	Define $f_i : U_i \rightarrow \realset$ by
	\[ f_i(x_1, \dots, x_n) = c_i - \alpha_1 \frac{x_1^2}{2} - \cdots - \alpha_n \frac{x_n^2}{2}. \]
	The unique singular point of $f_i$ on $U_i$ is $p_i$. It satisfies $\Ind_{p_i}(f_i) = \Ind_{p_i}(w)$. Moreover, $-X_w \cdot f_i = \alpha_1^2 x_1^2 + \cdots + \alpha_n^2 x_n^2 \geq 0$.
	
	\begin{figure}
		\centering
		\def\svgwidth{0.9\textwidth}
		\small%
		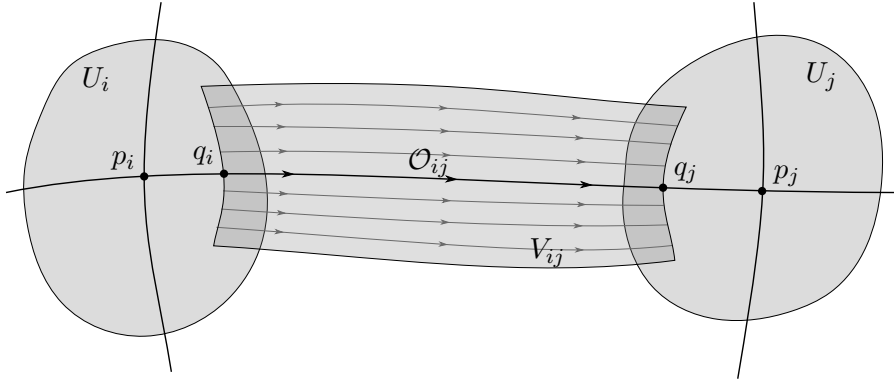
		\caption{Connecting $U_i$ and $U_j$ along the 1-dimensional orbit $\Ocal_{ij}$}
		\label{f:Morse_construction_step2}
	\end{figure}
	
	\textbf{Step 2.} Suppose $p_i \rightarrow p_j$. We then have $c_i < d_i^+ < d_j^- < c_j$. Let $\Ocal_{ij}$ be the $1$-dimensional orbit of $\rho$ joining $p_i$ and $p_j$. There exists $q_i \in U_i \cap \Ocal_{ij}$, $q_j = \varphi_w^T(q_i) \in U_j \cap \Ocal_{ij}$, $V_{ij} \subset M$, a neighborhood $\Omega$ of $0$ in $\realset^{n-1}$ and a diffeomorphism $\varphi_{ij} : V_{ij} \rightarrow \Omega \times [0,1]$ such that
	\begin{itemize}
		\item $\varphi_{ij}^{-1}(\Omega \times \set{0}) \subset \set{p \in U_i \mid f_i(p) = d_i^+}$,
		\item $\varphi_{ij}^{-1}(\Omega \times \set{1}) \subset \set{p \in U_j \mid f_j(p) = d_j^-}$,
		\item $\varphi_{ij}^{-1}(\set{0} \times [0,1]) = \set{\varphi_w^t(q_i) \mid 0 \leq t \leq T}$,
		\item $(\varphi_{ij})_\ast (-X_w) = \partdiff{}{t}$ where $(s_1, \dots, s_{n-1},t)$ are coordinates on $\Omega \times [0,1]$
	\end{itemize}
	(see Figure~\ref{f:Morse_construction_step2}).
	
	Let $f_{ij} : U_{ij} \rightarrow \realset$ be defined by $f_{ij} \circ \varphi_{ij}^{-1}(s,t) = d^+_i + t(d^-_j - d^+_i)$. Then $f_{ij}$ has no singular points on $U_{ij}$ and $-X_w \cdot f_{ij} \geq \varepsilon_{i,j} > 0$.
	
	\begin{figure}
		\centering
		\def\svgwidth{\textwidth}
		\tiny%
		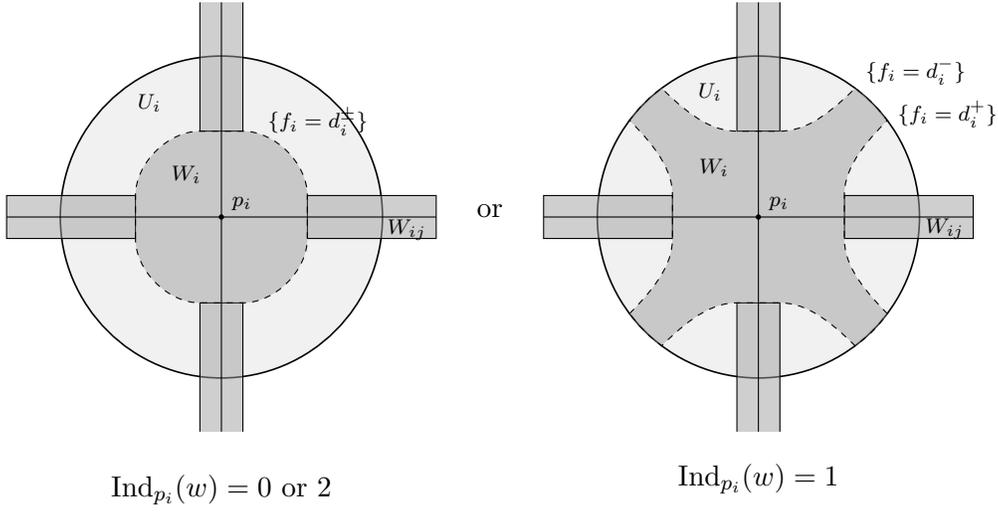
		\caption{Restricting $U_i$ to $W_i$ in dimension 2}
		\label{f:Morse_construction_step3}
	\end{figure}
	
	\textbf{Step 3.} Without loss of generality, assume $U_i \cap U_j = \emptyset$ whenever $i \neq j$, and $U_{ij} \cap U_{kl} = \emptyset$ whenever $\set{i,j} \neq \set{k,l}$. Define
	\[ W_i = \set{p \in U_i \mid d_i^- < f_i(p) < d_i^+} \]
	and $W_{ij}$ as the interior of $V_{ij}$ (see Figure~\ref{f:Morse_construction_step3}). Let $W = (\cup_i W_i) \cup (\cup_{ij} W_{ij})$. The maps $f_i$ and $f_{ij}$ define a continuous function on $\closure{W}$ which is smooth on $W$.
	
	\textbf{Step 4.} Let $W'$ be a connected component of $M \setminus \closure{W}$. It lies in some $n$-dimensional orbit $\Ocal$ of the action. Extends $f$ continuously on $W \cup \closure{W'}$ using the flow $\varphi_w^t$.
	
	For $p \in \boundary{W'}$, define
	\[ T(p) = \inf \set{t > 0 \mid \varphi_t(p) \notin W'}. \]
	The dynamics of $\varphi_w^t$ are such that for any $q \in W'$, there exists a unique $p \in \boundary{W'}$ such that $q = \varphi_w^t(p)$, with $0 < t < T(p)$. Set
	\[ f(q) = \frac{T(p) - t}{T(p)} f(p) + \frac{t}{T(p)} f(\varphi_w^{T(p)}(p)) \]
	(using that both $p$ and $\varphi_w^{T(p)}(p)$ lie in $\closure{W}$ where $f$ is already defined). Now $f$ is continuous on $\closure{W} \cup \closure{W'}$, and smooth on $W$ and $W'$. Moreover, by definition of $f$ on $W'$ we get :
	\[ (X_w \cdot f)(q) = \frac{f(\varphi_w^{T(p)}(p)) - f(p)}{T(p)}. \]
	But by construction of $W$, $p$ has to lie in $\closure{U_i}$ around some $p_i$, $\varphi_w^t(p)$ lies un $\closure{U_j}$ for some $p_j$, and there exists a sequence of fixed points $p_i \rightarrow \cdots \rightarrow p_j$.
	So finally
	\[ (X_w \cdot f)(q) = \frac{d^-_j - d^+_i}{T(p)} > 0. \]
	
	\textbf{Step 5.} We now have a function $f : M \rightarrow \realset$ continuous on $M$, smooth on a finite family of open sets $(W_\lambda)_{\lambda \in \Lambda}$ such that $\cup_\lambda \closure{W_\lambda} = M$. Also, the restriction of $f$ to $W_\lambda$ can be extended to a smooth function on an open subset containing $\closure{W_\lambda}$. Everywhere it is defined, we have $(-X_w \cdot f) \geq 0$. More precisely, $(-X_w \cdot f) \geq \varepsilon > 0$ outside some neighborhoods around the fixed points of $\rho$. Now it suffices to smooth the function $f$ with some operator defined below, the rest of the proof is detailed in Proposition~\ref{p:smoothing_piecewise_Morse}.
\end{proof}

We will use the smoothing operator introduced by de Rham in~\cite{derham2012differentiable}. Let us recall here its definition.

Start with a smooth strictly increasing function $\rho : [0, 1) \rightarrow [0, \infty)$ such that $\rho(t) = t$ when $t < 1/3$, and $\rho(t) \geq t$ for any $t \geq 0$. This function determines a radial smooth diffeomorphism $h : B_1 \rightarrow \realset^n$ from the unit ball in $\realset^n$ to $\realset^n$ by:
\[ h(x) = \frac{\rho(\norm{x})}{\norm{x}} x. \]
Then define the smooth diffeomorphism $\sigma_v : \realset^n \rightarrow \realset^n$ by
\[ \sigma_v(x) = \begin{cases}
h^{-1}(h(x) + v) &\text{if } \norm{x} < 1, \\
x &\text{if } \norm{x} \geq 1.
\end{cases} \]
Also, fix a smooth positive function $\chi : \realset^n \rightarrow \realset$ with support included in $\closure{B_1}$ and such that
\[ \int_{\realset^n} \chi(v) \dd v = 1. \]
Suppose $\Omega$ is an open subset of $\realset^n$ containing $\closure{B_1}$ (thus $\sigma_v(\Omega) = \Omega$). Define for any $t > 1$ the regularization operator $R_t$ as follows. If $g : \Omega \rightarrow \realset$ is a continuous function, set:
\[ \forall x \in \Omega, \ (R_t g)(x) = \int_{\realset^n} g \circ \sigma_v(x) t^n \chi(tv) \dd v. \]
The resulting function $R_t g : \Omega \rightarrow \realset$ is smooth.

\begin{lemma}
	\label{l:smoothing_inequality}
	Suppose $\closure{\Omega}$ is compact and let $g : \Omega \rightarrow \realset$ be a continuous function. Suppose that, for some $s_0 > 0$, $g$ extends to a Lipschitz continuous function on $\Omega_{s_0} = \Omega + s_0 B_1$ satisfying for any $x \in \Omega$ and $0 \leq s < s_0$,
	\[ g(x + se_1) - g(x) \geq s \varepsilon \]
	(where $e_1 = (1, 0, \dots, 0)$ denotes the first vector in the canonical basis of $\realset^n$).
	
	Then for any $0 < \varepsilon' < \varepsilon$, there exists $t_0 > 0$ such that for any $t \geq t_0$, the smoothed function $R_t g : \Omega \rightarrow \realset$ satisfies
	\[ \partdiff{R_t g}{x_1}(x) \geq \varepsilon' \]
	for any $x \in \Omega$.
\end{lemma}

\begin{proof}
	Fix $0 < \varepsilon' < \varepsilon$. For any $x \in \Omega$ and $0 \leq s < s_0$,
	\[ (R_t g)(x + s e_1) - (R_t g)(x) = \int_{\realset^n} (g \circ \sigma_v(x + s e_1) - g \circ \sigma_v(x)) t^n \chi(t v) \dd v.  \]
	Recall that $\chi$ has support included in $B_1$, so the above integral can actually be computed on $B_{1/t}$.
	
	Denote by $k$ the Lipschitz constant of $g$. The map $F : \closure{\Omega} \times [0,s_0] \times \closure{B_1} \rightarrow \realset^n$ defined by
	\[ F(x,s,v) = \sigma_v(x + s e_1) - s e_1 \]
	is smooth. By Heine--Cantor theorem, there exists $\eta > 0$ such that for any $x \in \Omega$, $s \in [0, s_0]$ and $v \in \closure{B_1}$, we have
	\[ \Norm{\partdiff{F}{s}(x,s,v)} = \Norm{\partdiff{F}{s}(x,s,v) - \partdiff{F}{s}(x, s, 0)} \leq \frac{\varepsilon - \varepsilon'}{k} \]
	whenever $\norm{v} \leq \eta$. Applying the mean value theorem to $s \mapsto F(x,s,v)$, we obtain
	\[ \norm{\sigma_v(x + s e_1) - s e_1 - \sigma_v(x)} \leq \frac{s (\varepsilon - \varepsilon')}{k} \]
	for any $x \in \Omega, s \in [0,s_0]$ and $v \in \closure{B_1}$ such that $\norm{v} \leq \eta$. It follows that
	\[ \abs{g \circ \sigma_v(x + s e_1) - g(\sigma_v(x) + s e_1)} \leq s (\varepsilon - \varepsilon'). \]
	
	Moreover, recall that it is assumed that for any $x \in \Omega$, $0 \leq s < s_0$ and $v \in \realset^n$ we have
	\[ g(\sigma_v(x) + s e_1) - g \circ \sigma_v(x) \geq s \varepsilon, \]
	so finally, when $\norm{v} \leq \eta$, the following inequality holds:
	\[ g \circ \sigma_v(x + s e_1) - g \circ \sigma_v(x) \geq s \varepsilon'. \]
	It follows that for any $t \geq t_0 = 1 / \eta$, $x \in \Omega$, $0 \leq s < s_0$,
	\[ (R_t g)(x + s e_1) - (R_t g)(x) \geq s \varepsilon'. \]
	We finish the proof by dividing both sides by $s$ and taking the limit $s \to 0$.
\end{proof}

The above lemma will be used together with the following result:
\begin{lemma}
	\label{l:continuous_to_lipschitz}
	Suppose now $\Omega_{s_0}$ is convex and bounded and let $g : \Omega_{s_0} \rightarrow \realset$ be a continuous function. Suppose there exist finite families $(\Omega_i)_{i \in I}$ and $(\Omega'_i)_{i \in I}$ of open sets of $\realset^n$ and a family $(\tilde{g}_i : \Omega'_i \rightarrow \realset)_{i \in I}$ of smooth functions such that:
	\begin{itemize}
	\item $\cup_{i \in I} \closure{\Omega_i}$ covers $\Omega_{s_0}$,
	\item for each $i \in I$, $\closure{\Omega_i} \subset \Omega'_i$,
	\item for each $i \in I$ $\tilde{g}_i$ and $g$ coincide on $\closure{\Omega_i}$.
	\end{itemize}
	Then the function $g$ is Lipschitz continuous on $\Omega_{s_0}$.
\end{lemma}

\begin{proof}
	For each $i \in I$, fix an open set $U_i$ such that $\closure{\Omega_i} \subset U_i \subset \closure{U_i} \subset \Omega'_i$ and $\closure{U_i}$ is compact (without loss of generality, we can suppose $\Omega'_i$ is bounded). Then
	\[ k = \max_{i \in I} \left( \sup \set{ \norm{\dd \tilde{g}_i(x)} \mid x \in \closure{U_i} } \right) \]
	is finite. Let $x_0, x_1$ be two distinct points in $\Omega_{s_0}$. Since $\Omega_{s_0}$ is convex, for any $t \in [0,1]$, the point
	$ x_t = (1 - t)x_0 + t x_1 $
	lies in $\Omega_{s_0}$. Define $T \in [0, 1]$ by
	\[ T = \sup \set{ t \in [0,1] \mid \norm{g(x_t) - g(x_0)} \leq k \norm{x_t - x_0} }. \]
	We proceed using reductio ad absurdum and suppose $T < 1$. For any $n > 1/(1 - T)$, there exists $i_n \in I$ such that $x_{T+ 1/n} \in \closure{\Omega_{i_n}}$. Since $I$ is finite, there exists a fixed $i \in I$ and a subsequence $(\varepsilon_n)_n$ of $(1/n)_n$ such that $x_{T + \varepsilon_n} \in \closure{\Omega_{i}}$ for any $n$. In particular, $x_T = \lim_{n \to \infty} x_{T + \varepsilon_n}$ also lies in $\closure{\Omega_{i}} \subset U_{i}$. Let $r > 0$ be such that the open ball $B(x_T, r)$ of center $x_T$ and radius $r$ is included in the open set $U_i$. Then for any $\delta < r/\norm{x_1 - x_0}$, the point $x_{T + \delta}$ lies in $U_i$, and we can apply the mean value theorem to the function $s \mapsto \tilde{g}_{i}(x_s)$ to obtain
	\[ \norm{\tilde{g}_{i}(x_{T+\delta}) - \tilde{g}_{i}(x_T) } \leq k \norm{x_{T+\delta} - x_{T}}. \]
	In particular for $\delta = \varepsilon_n$, since $x_{T+\varepsilon_n}$ and $x_T$ lie in $\closure{\Omega_{i}}$ where $g$ and $\tilde{g}_i$ coincide, we get for large enough $n$:
	\[ \norm{g(x_{T+\varepsilon_n}) - g(x_T)} \leq k \norm{x_{T+\varepsilon_n} - x_T}. \]
	But by definition of $T$, we have $\norm{g(x_T) - g(x_0)} \leq k \norm{x_T - x_0}$. It follows that
	\[ \norm{g(x_{T+\varepsilon_n}) - g(x_0)} \leq k (\norm{x_{T+\varepsilon_n} - x_T} + \norm{x_T - x_0}) = k \norm{x_{T + \varepsilon_n} - x_0}, \]
	which contradicts the definition of $T$ as a supremum.
	
	It follows that $T = 1$, that is $\norm{g(x_1) - g(x_0)} \leq k \norm{x_1 - x_0}$.
\end{proof}

This operator $R_t$ can then be used to define local regularization operators on $M$. Suppose $(V, \varphi)$ is a local chart on $M$ such that $\closure{B_1} \subset \varphi(V)$. Define the regularization operator $S_{t, V}$ acting on continuous functions $f : M \rightarrow \realset$ by:
\[ (S_{t, V}f)(p) = \begin{cases}
	R_t(f \circ \varphi^{-1})(\varphi(p)) &\text{if } p \in V, \\
	f(p) &\text{otherwise}.
\end{cases} \]
The resulting function $S_{t, V}f : M \rightarrow \realset$ is smooth on $V$.

We are now able to conclude the proof of Theorem~\ref{t:construction_of_Morse_functions}, by applying the following proposition.

\begin{proposition}
	\label{p:smoothing_piecewise_Morse}
	Let $f : M \rightarrow \realset$ be a continuous function on a compact $n$-dimensional manifold $M$. Suppose there exists a finite family of disjoint open subsets $(W_i)_{i \in I}$, a family $(K_i)_{i \in I}$ of (possibly empty) compact subsets $K_i \subset W_i$, a smooth vector field $X$ with (complete) flow $\Phi^t$ on $M$ and $\varepsilon > 0$ such that:
	\begin{enumerate}
		\item the restriction of $f$ to each $W_i$ is a smooth Morse function (possibly with no singularities),
		\item the set $\Scal = \cup_i W_i$ on which $f$ is smooth satisfies $\closure{\Scal} = M$,
		\item for each $i \in I$, there exists an open subset $W'_i \subset M$ containing $\closure{W_i}$ and a smooth function $\tilde{f}_i : W'_i \rightarrow M$ such that $f$ and $\tilde{f}_i$ coincide on $\closure{W_i}$,
		\item for any $p \in \Scal$, $\Phi^t(p) \in \Scal$ except for a finite number of $t \in \realset$,
		\item $(X \cdot f)(p) \geq \varepsilon$ for any $p \in \Scal \setminus (\cup_i K_i)$. In particular, the singularities of $f$ and $X$ lie in $\cup_i K_i$.
	\end{enumerate}
	Then there exists a smoothing operator $S_t$ on $M$ such that for any large enough $t > 0$, $S_t f : M \rightarrow \realset$ is a Morse function whose singularities are exactly the Morse singularities of $f$ with the same Morse indices.
\end{proposition}

\begin{proof}
	Denote by $\Delta$ the set of continuous functions $F : M \rightarrow \realset$ such that:
	\begin{itemize}
		\item $F$ is smooth on each $W_i$,
		\item for each $i \in I$, there is an open subset $W'_i \subset M$ and a smooth function $\tilde{F}_i : W'_i \rightarrow M$ such that $F$ and $\tilde{F}_i$ coincide on $\closure{W_i} \subset W'_i$,
		\item there exists $\varepsilon' > 0$ such that $(X \cdot F)(q) \geq \varepsilon'$ for any $q \in S \setminus (\cup_i K_i)$.
	\end{itemize}
		
	Let $p \in M \setminus \Scal$ be a point where $F$ might not be smooth. Since $X(p) \neq 0$, there exists a local chart $(U_p, \varphi_p)$ around $p$ such that
	\[ (\varphi_p)_\ast X = \alpha_p \partdiff{}{x_1} \text{ with } \alpha_p > 0. \]
	Without loss of generality, we can suppose that $\closure{U_p} \cap K_i = \emptyset$ for any $i \in I$. Also, up to multiplying $\varphi_p$ by some constant $c \geq 1$, we can assume that $\varphi_p(U_p)$ contains $\closure{B_1}$. Take $V_p \subset U_p$ and $s_p > 0$ such that $\Omega_p = \varphi_p(V_p)$ is convex, bounded and contains $\closure{B_1}$, and such that $\Omega_{p,s_p} = \varphi_p(V_p) + s_p B_1 \subset \varphi_p(U_p)$.
	
	Fix $F \in \Delta$. The function $g = F \circ \varphi_p^{-1}$ is continuous on $\Omega_{p,s_p}$. For each $i \in I$, $g$ is smooth on 
	\[ \Omega_i = \varphi_p(V_p \cap W_i), \]
	and extends to a smooth function $\tilde{g}_i = \tilde{F}_i \circ \varphi_p^{-1}$ on 
	\[ \Omega'_i = \varphi_p(U_p \cap W'_i). \] Then, by Lemma~\ref{l:continuous_to_lipschitz}, $g$ is Lipschitz continuous on $\Omega_{p,s_p}$.
	
	Let $q \in V_p \cap \Scal$ be a point where $F$ is smooth. Fix $0 < s < s_p$. There exists a finite sequence $0 = s_0 < s_1 < \cdots < s_k = s$ such that $\Phi^{s'}(q) \in \Scal$ for any $s_j < s' < s_{j+1}$. Since $\closure{U_p} \cap (\cup_i K_i) = \emptyset$, for such a $s'$ we have $(X \cdot F)(\Phi^{s'}(q)) \geq \varepsilon'$. It follows that for any $s_j < a < b < s_{j+1}$ we have
	\[ F(\Phi^b(q)) - F(\Phi^a(q)) \geq (b-a)\varepsilon'. \]
	By continuity, the result holds for $(a, b) = (s_j, s_{j+1})$. Summing those inequalities we obtain
	\[ F(\Phi^s(q)) - F(q) \geq s \varepsilon'. \]
	The result extends to any $q \in \closure{V_p \cap \Scal} = V_p$. In other words, the function $g = f \circ \varphi_p^{-1}$ satisfies for any $x \in \Omega_p = \varphi_p(V_p)$ and $0 < s < s_p$:
	\[ g(x + se_1) - g(x) \geq s \varepsilon' / \alpha_p. \]
	Let $S_{t,p}$ be the regularization operator defined on $V_p$. Fix $\varepsilon'' < \varepsilon'$. By Lemma~\ref{l:smoothing_inequality}, there exists $t_p > 0$ such that
	\[ \partdiff{R_t g}{x_1} \geq \varepsilon''/\alpha_p  \]
	for any $t \geq t_p$ and $x \in \Omega_p$, or equivalently,
	\[ (X \cdot S_{t,p} F)(q) \geq \varepsilon'' \]
	for any $q \in V_p$. Since $(S_{t,p} F)(q) = F(q)$ for any $q \in M \setminus V_p$, it follows that $S_{t,p}F$ is in $\Delta$.
	
	By compactness, there exists a finite number of points $p_1, \dots, p_k \in M$ such that the open subsets $V_{p_j}$ cover $M \setminus \Scal$. Define the total regularization operator
	\[ S_t = S_{t, p_k} \circ S_{t, p_{k-1}} \circ \cdots \circ S_{t, p_1}. \]
	Then $S_t f$ is smooth on $M$ and coincide with $f$ on each $K_i$. Moreover, since $f \in \Delta$, for any $t \geq \max \set{t_{p_j} \mid 1 \leq j \leq k}$ the function $S_t f$ is also in $\Delta$. In particular, there exists $\varepsilon' > 0$ such that
	\[ (X \cdot S_t)(q) \geq \varepsilon' \]
	for any $q \in \Scal \setminus (\cup_i K_i)$. By continuity, this holds for any $q \in M \setminus (\cup_i K_i)$, so in particular $S_t f$ has no singularities outside $\cup_i K_i$.
\end{proof}

As an immediate corollary of Theorem~\ref{t:construction_of_Morse_functions}, we obtain the Morse inequalities \cite{milnor1963morse}:

\begin{corollary}
	\label{c:Morse_inequalities}
	Let $\rho$ be a totally hyperbolic action of $\realset^n$ on a compact $n$-manifold $M$. Suppose there exists $w \in \realset^n$ such that $\varphi^t_w$ has no cycles. For $0 \leq i \leq n$, denote by $c_i$ the number of fixed points of $\rho$ of index $i$ with respect to $w$ and by $b_i(M)$ the $i$-th Betti number of $M$. Then we have the inequalities:
	\[ c_i - c_{i-1} + \cdots + (-1)^i c_0 \geq b_i(M) - b_{i-1}(M) + \cdots + (-1)^i b_0(M) \]
	for all $0 \leq i \leq n$, with equality when $i = n$, that is
	\[ \sum_{i=0}^n (-1)^i c_i = \chi(M). \]
\end{corollary}

In particular, for all $0 \leq i \leq n$, we have the weak Morse inequalities
\[ c_i \geq b_i(M), \]
so the number $V$ of fixed points of $\rho$ is at least $\sum_{i=1}^n b_i(M)$.
For instance, if $M$ is a compact orientable two-dimensional surface of genus $g$, then $V \geq 2g + 2$.

\section{Number of hyperbolic domains}
\label{s:number_of_domains}

Let us begin with the following observation, which is an immediate consequence of the study of the quasi Morse-Smale flows generated by generic vectors.
\begin{theorem}
	\label{t:number_of_domains_dimension_n}
	The number of hyperbolic domains of a totally hyperbolic action of $\realset^n$ on a compact connected $n$-manifold $M$ is equal to $k.2^n$, where $k$ is the number of attractive (or repulsive) fixed points of the quasi Morse--Smale flow $\varphi_w^t$ generated by a generic $w \in \realset^n$. In particular, this number $k$ does not depend on the choice of a generic $w$.
\end{theorem}

\begin{proof}
	Around a fixed point $p \in M$ of the action $\rho$ there are exactly $2^n$ hyperbolic domains. The local normal form theorem for totally hyperbolic actions implies that if $p$ is an attractive point of the flow $\varphi_w^t$ on the closure of some hyperbolic domain $\Ocal$, then it is locally an attractive point in neighborhood $U$ of $p$. In particular, it is an attractive point on the closure of any other hyperbolic domain around it. Since every hyperbolic domain of $\rho$ has a unique attractive fixed point on its boundary, it follows that there are exactly $k.2^n$ domains, where $k$ is the total number of attractive points of the flow $\varphi_w^t$ on $M$. Of course, the same proof can be done by considering repulsive points instead.
\end{proof}

\begin{remark}
	With the notation of Corollary~\ref{c:Morse_inequalities}, the above theorem implies that $c_0$ and $c_n$ are equal and do not depend on $w$.
\end{remark}

\begin{remark}
	\label{r:number_of_domains_2-sphere}
	When $M$ is the two-dimensional sphere $\sphere{2}$, we have a more precise result: the number of hyperbolic domains is a multiple of $8$. Indeed, it is equal to $4c_n$ by Theorem~\ref{t:number_of_domains_dimension_n}, where $c_n$ is the number of attractive points of some quasi Morse--Smale flow $\varphi_w^t$. But a point $p$ at the intersection of two loops $H_i$ and $H_j$ is attractive if and only if $w$ is in the convex cone $C_{ij}$ spanned by $v_i$ and $v_j$. It follows that
	\[ c_n = \sum_{\substack{1 \leq i < j \leq N \\ w \in C_{ij}}} \card (H_i \cap H_j). \]
	On the sphere, two closed loops intersect an even number of times so the above sum is even.
\end{remark}

Note that, as $w$ varies in $\realset^n$, a fixed point $x$ of $\rho$ takes all the possible states (attractive, repulsive or saddle) with respect to the flow $\varphi_w^t$ on $M$. So there exist different ``jigsaw puzzle'' decompositions of $M$ whose pieces are the union of $2^n$ hyperbolic domains around a fixed point.

\begin{figure}
	\centering
	\def\svgwidth{0.99\textwidth}
	\small%
	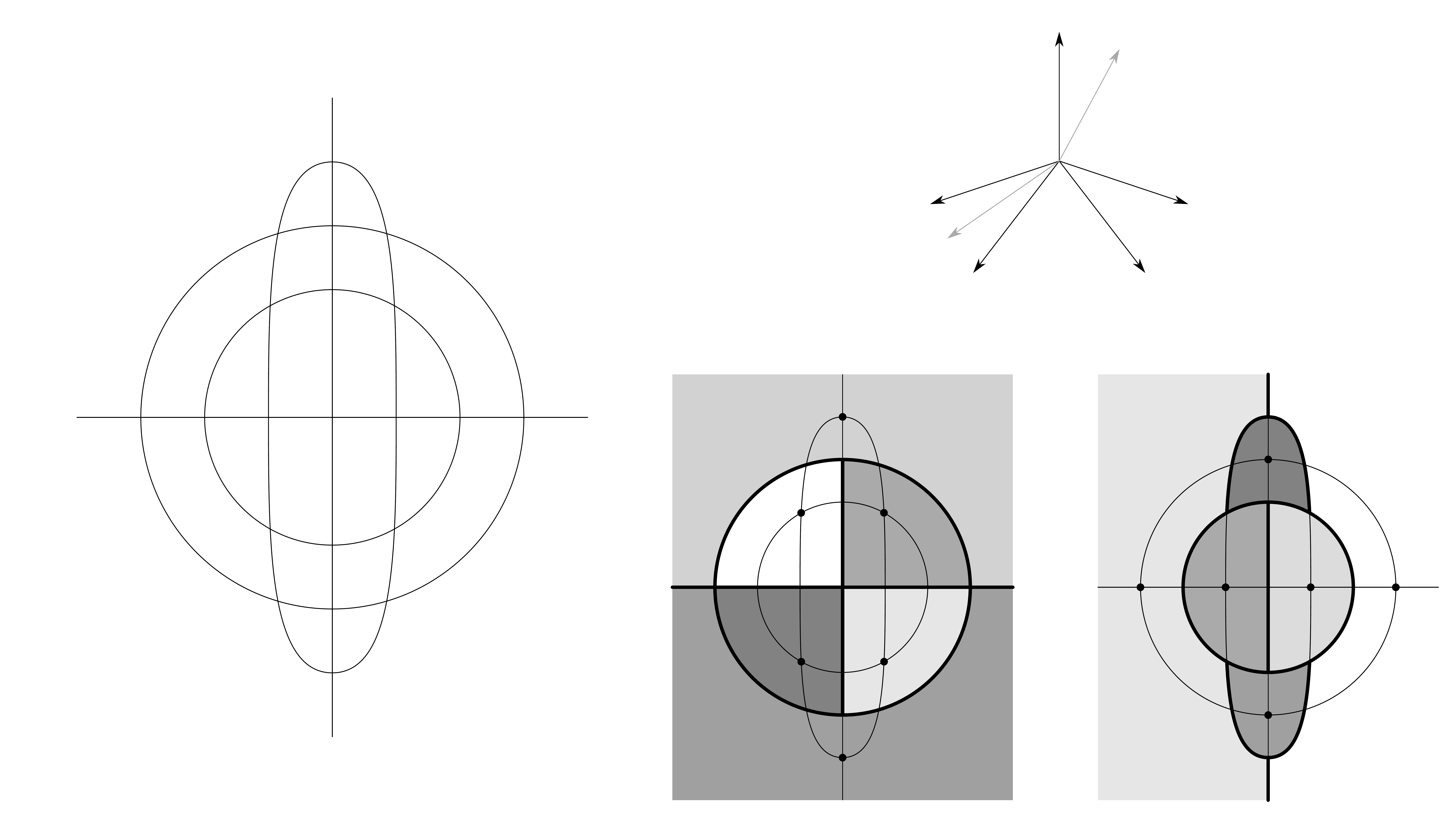
	\caption{Example of two different jigsaw puzzle decompositions for some action of $\realset^2$ on $\sphere{2}$}
	\label{f:jigsaw_puzzles}
\end{figure}

Figure~\ref{f:jigsaw_puzzles} shows two different such decompositions on an example. On the upper right of the figure is a complete fan in $\realset^2$ defined by five vectors $v_1, \dots, v_5$, corresponding to a totally hyperbolic action of $\realset^2$ on $\sphere{2}$. Next to it is the associated decomposition of $\sphere{2}$ into hyperbolic domains delimited by hypersurfaces $H_1, \dots, H_5$, after a stereographic projection from an intersection point between $H_3$ and $H_4$. Below this are two jigsaw puzzles decomposition of $\sphere{2}$ induced by the choice of two vectors $w_A$ and $w_B$ in different 2-dimensional cones in the fan. The emphasized vertices are the attractive points of the corresponding 1-dimensional flow on the sphere.

\begin{figure}
	\centering
	\def\svgwidth{0.9\textwidth}
	\small%
	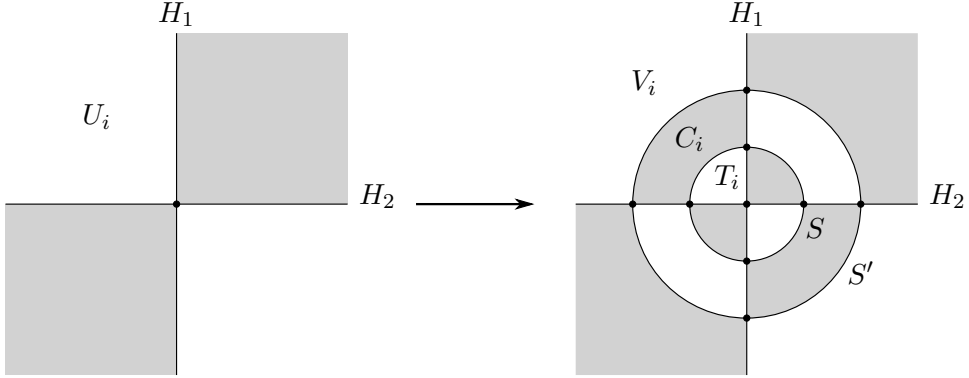
	\caption{Adding two spheres around a fixed point}
	\label{f:spheres_around_fixed_point}
\end{figure}

We are now interested in the construction of actions with a given number of hyperbolic domains.
Let us start with the following useful result.
\begin{proposition}
	\label{p:adding_spheres_around_fixed_point}
	Let $\rho : \realset^n \times M \rightarrow M$ be a totally hyperbolic action of $\realset^n$ on a $n$-dimensional compact manifold $M$. Let $x \in M$ be a fixed point of the action. It is the intersection of $n$ hypersurfaces $H_1, \dots, H_n$ corresponding to $(n-1)$-dimensional orbits of the action, and these hypersurfaces delimit $2^n$ hyperbolic domains $U_1, \dots, U_{2^n}$ around $x$. Then
	\begin{enumerate}
		\item the new decomposition of $M$ obtained by adding two small enough concentric $(n-1)$-spheres centered at $x$ can be realized as the hyperbolic domains of a totally hyperbolic action of $\realset^n$ on $M$,
		\item conversely, if there are two concentric $(n-1)$-spheres centered at $x$ intersecting only $H_1, \dots, H_n$ as in Figure~\ref{f:spheres_around_fixed_point}, then the new decomposition of $M$ obtained by removing these spheres can be realized as the hyperbolic domains of a totally hyperbolic action of $\realset^n$ on $M$.
	\end{enumerate}
\end{proposition}

\begin{proof}
	(1) Adding two spheres $S$ and $S'$ around $x$ splits each domain $U_i$ into three domains $T_i$, $C_i$ and $V_i$ which are ``curved polytopes'' with respectively $n+1$, $n+2$ and $f+1$ faces, where $f$ is the number of faces of $U_i$ (see Figure~\ref{f:spheres_around_fixed_point} for an illustration in dimension 2). The other domains remain unchanged.
	
	Let $v_1, \dots, v_N$ be vectors in $\realset^n$ generating a complete fan $F$ corresponding to the hyperbolic action $\rho$, indexed in such a way that the vector $v_i$ corresponds to the hypersurface $H_i$ when $1 \leq i \leq n$. We want to construct two vectors $w$ and $w'$, associated to $S$ and $S'$ respectively, such that the complete fan $F'$ generated by $v_1, \dots, v_N, w, w'$ is compatible with the new decomposition of $M$. Take $w'$ in the cone spanned by $v_1, \dots, v_n$. Then the compatibility of $F$ with $U_i$ implies that $F'$ will be compatible with $V_i$. Now set $w = -w'$ (or a vector close enough to $-w'$ if the latter coincide with some $v_i$), and the fan $F'$ is compatible with the $T_i$ and $C_i$.
	
	(2) With the same notation as above, remark first that the domains $U_i$ obtained when removing the spheres $S$ and $S'$ have at least three corners, that is are not only bounded by the hypersurfaces $H_1, \dots, H_n$. Indeed, suppose the converse is true, then it means that before removing the spheres, the domain $V_i$ was only bounded by $H_1, \dots, H_n$ and $S'$. Then by compatibility of the fan $F'$ with the domain $C_i$, it follows that either $v_{N+1}$ or $v_{N+2}$ is not in the convex cone spanned by $v_1, \dots, v_n$. But that implies that $F'$ is not compatible either with $T_i$ or $V_i$, which contradicts our assumption. It follows that $U_i$ is a curved polytope.
	
	Now we have to show that the fan $F$ obtained by removing $v_{N+1}$ and $v_{N+2}$ from the fan $F'$ is a complete fan compatible with this new decomposition of $M$. It suffices to check the compatibility with each domain $U_i$ since the other domains are unchanged or deleted. The compatibility of $F$ with the combinatorics of the faces of $U_i$ is guaranteed by the previous compatibility of $F'$ with the domain $V_i$. It remains to check that the cones in the sub-fan of $F$ corresponding to $U_i$ are all convex: this comes from the compatibility of the fan $F'$ with the faces $C_i$ and $T_i$.
\end{proof}

With this proposition, once an example of hyperbolic action has been given, it is possible to construct new hyperbolic actions with more hyperbolic domains on the same manifold $M$.

\begin{corollary}
	\label{c:augment_domains}
	Let $M$ be a $n$-dimensional compact manifold. Suppose there exists a totally hyperbolic action $\rho$ on $M$.
	\begin{enumerate}
		\item If $D$ is the number of hyperbolic domains of $\rho$, then for any $k \geq 0$, there exists a totally hyperbolic action on $M$ with $D'=D +k2^{n+1}$ hyperbolic domains.
		\item Let $p$ be the number of faces of some hyperbolic domain of $\rho$. Then for any $q \geq p$, there exists a totally hyperbolic action on $M$ admitting a hyperbolic domain with $q$ faces.
	\end{enumerate}
\end{corollary}

\begin{proof}
	It suffices to apply recursively the construction given by Proposition~\ref{p:adding_spheres_around_fixed_point}.
	
	For (1), apply the construction to any fixed point of $\rho$. Each one of the $2^n$ domains $U_i$ is split into exactly three domains, so the new decomposition has $2 \times 2^n$ more hyperbolic domains.
	
	For (2), apply the construction to a fixed point of $\rho$ which is the corner of the domain $U$ with $p$ faces. We saw that $U$ is split into three domains $T$, $C$ and $V$, where the domain $V$ has $p+1$ faces.
\end{proof}

Let us apply this to the case when $M$ is a closed surface. This will not only provide information on the number of hyperbolic domains, but also on the number of one-dimensional orbits and fixed points of the action. Indeed, in dimension 2, the decomposition of a surface $\Sigma$ into orbits of a totally hyperbolic action $\rho : \realset^2 \times \Sigma \rightarrow \Sigma$ can be seen as the embedding of some graph $\Gamma$ in $\Sigma$, in such a way that the fixed points, one-dimensional orbits and hyperbolic domains of $\rho$ correspond respectively to the vertices, edges and faces of the embedded graph $\Gamma$.

First, since the graph $\Gamma$ is embedded in $\Sigma$, it has to satisfy the well-known Euler's formula
\[ V - E + F = \chi(\Sigma), \]
where $V$, $E$ and $F$ are respectively the number of vertices, edges and faces of the embedded graph $\Gamma$, and $\chi(\Sigma)$ is the Euler characteristic of the surface $\Sigma$.

Moreover, the local structure of fixed points of a totally hyperbolic action implies that each vertex corresponds to an intersection between exactly two one-dimensional orbits. Then each vertex is adjacent to exactly 4 edges, which are necessarily distinct for the faces to be simply connected ``curved polygons''. Thus $\Gamma$ is what is a called a 4-valent simple graph.

\begin{proposition}
	\label{p:graph_is_4_valent}
	Let $\Gamma$ be the embedded graph induced by a totally hyperbolic action of $\realset^2$ on a surface $\Sigma$.
	
	Then $\Gamma$ is a 4-valent simple graph. In particular, it satisfies the following identities:
	\[ \begin{cases}
	E = 2V \\
	F = V + \chi(\Sigma)
	\end{cases} \]
\end{proposition}

\begin{proof}
	The first identity is a classical result in graph theory, of which we recall the proof here. Because the graph is 4-valent, the sum
	\[ \sum_{v \text{ vertex of } \Gamma} (\text{number of edges adjacent to } v) \]
	is equal to $4V$. Then remark that each edge of $\Gamma$ contributes to the sum exactly twice (once for each of its end points) so the above sum is also equal to $2E$. The second identity follows using Euler's formula.
\end{proof}

We are now able to prove that, in dimension 2, any number of domains compatible with Theorem~\ref{t:number_of_domains_dimension_n} and Remark~\ref{r:number_of_domains_2-sphere} is realizable.

\begin{proposition}
	\label{p:realization_number_of_domains_on_surface}
	Let $k \geq 1$.
	\begin{itemize}
		\item The $2$-sphere $\sphere{2}$ admits a totally hyperbolic action with $8k$ hyperbolic domains.
		\item Any closed oriented surface of genus $g \geq 1$ admits a totally hyperbolic action with $4k$ hyperbolic domains.
		\item Any closed non-orientable surface admits a totally hyperbolic action with $4k$ hyperbolic domains.
	\end{itemize}
\end{proposition}

\begin{proof}
	Consider the examples of totally hyperbolic actions on closed surfaces given in~\cite{camacho1973morse} and~\cite{zung2014geometry}: those are examples with $8$ hyperbolic domains in the case of the sphere, and $4$ hyperbolic domains in the case of any other closed surface. According to Corollary~\ref{c:augment_domains}, we can increase this number by $8$ indefinitely, which prove the statement for the sphere. For the other closed surfaces, we have to prove that there also exist examples with $8$ hyperbolic domains in order to complete the proof.
	
	A decomposition into $8$ hyperbolic domains of a surface $\Sigma_g$ of genus $g \geq 0$ is given in~\cite{zung2014geometry}: embed $\Sigma_g$ in $\realset^3$ in such a way that it is symmetric with respect to the planes $\set{x=0}$, $\set{y=0}$ and $\set{z=0}$, and cut the surface along these planes. It splits $\Sigma_g$ into $8$ polygons with $g + 2$ sides, and to show that this decomposition can be realized as the hyperbolic domains of totally hyperbolic action of $\realset^2$ on $\Sigma_g$, it suffices to construct explicitly the action on one of these polygons and then extend it to the whole surface using reflections.
	
	\begin{figure}
		\centering
		\def\svgwidth{0.9\textwidth}
		\small%
		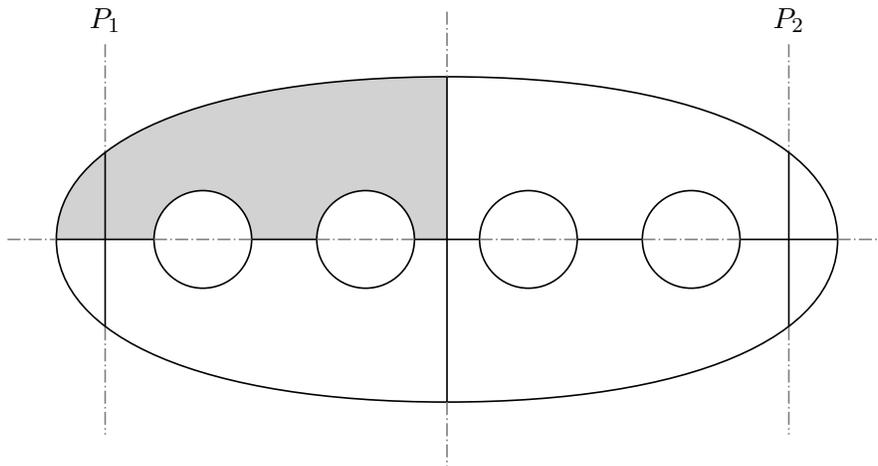
		\caption{Decomposition of a surface of genus $g \geq 0$ into $16$ domains}
		\label{f:sixteen_domains}
	\end{figure}
	
	Now, we will use the fact that any non-oriented closed surface $\Sigma'$ can be realized as the quotient of some $\Sigma_g$ embedded in $\realset^3$ as above by the involution $\sigma : (x,y,z) \mapsto (-x,-y,-z)$. Any totally hyperbolic action on $\Sigma_g$ which is invariant with respect to $\sigma$ descends to a totally hyperbolic action on $\Sigma'$. Consider again the surface $\Sigma_g$ embedded in $\realset^3$ in a symmetric way and cut along the three canonical planes, and now cut it again along two new parallel planes $P_1$ and $P_2$ as illustrated in Figure~\ref{f:sixteen_domains}. This gives a decomposition of $\Sigma_g$ into $16$ domains, and this decomposition is invariant with respect to $\sigma$. The colored part on the figure consists in two domains: a triangle $T$ and a $n$-gon $U$, and the edges of $T$ are issued from the same loops as three consecutive edges of $U$. Any fan compatible with $T$ can be completed into a fan compatible with $U$. It defines a totally hyperbolic action on the colored part, that we extend to $\Sigma_g$ using symmetries with respect to the three canonical planes. The hyperbolic domains of this action correspond exactly to the $16$ pieces of our decomposition, and since the involution $\sigma$ identifies these pieces pairwise, the corresponding hyperbolic action on the quotient manifold $\Sigma'$ has $8$ hyperbolic domains.
\end{proof}

\section{The case of the $2$-dimensional sphere}
\label{s:on_the_2_sphere}

In this section, we investigate the case where the surface $\Sigma$ is the 2-dimensional sphere $\sphere{2}$. In addition to the constants $V$, $E$ and $F$ defined previously, we denote by $N$ the number of closures of one-dimensional orbits of the action, and we recall that the latter are non-intersecting loops $L_1, \dots, L_N$ on $\sphere{2}$.

It is possible to color the faces of the plane graph $\Gamma$ with black and white in such a way that any two adjacent faces have different colors. Indeed, such a coloring can be constructed recursively. Start with a white sphere. According to the Jordan curve theorem, any loop $L_i$ splits the sphere $\sphere{2}$ into two connected components. Every time a loop is added to the sphere, leave one of the connected component unchanged, and swap the white and black colors on the other component.

\begin{figure}
	\centering
	\def\svgwidth{0.6\textwidth}
	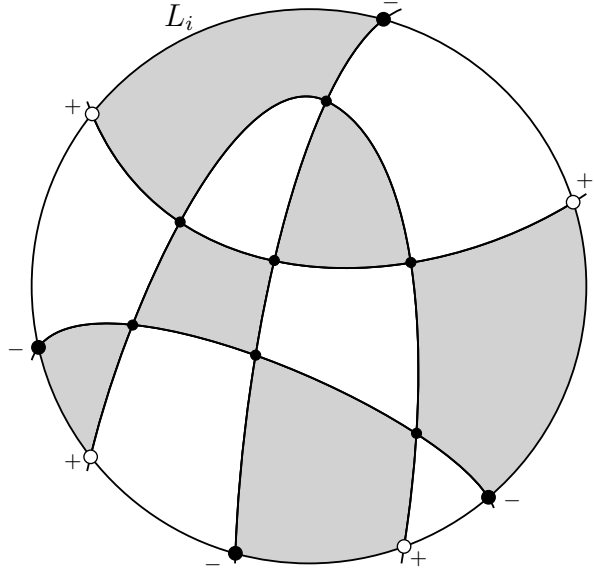
	\caption{Domains in a connected component of $\sphere{2} \setminus L_i$}
	\label{f:circle}
\end{figure}

\begin{proposition}
	\label{p:black_and_white_domains_on_whole_sphere}
	\label{p:black_and_white_domains_inside_circle}
	Suppose we have such a bi-coloring of the domains of $\rho$. For a subset $U \subset \sphere{2}$, denote by $F_B(U)$ (res. $F_W(U)$) the number of black domains (res. white domains) included in $U$. Then:
	\begin{enumerate}
		\item the numbers $F_B(\sphere{2})$ and $F_W(\sphere{2})$ are equal,
		\item if $U$ is on of the two connected components of the complement $\sphere{2} \setminus L_i$ of some loop, then $F_B(U)$ and $F_W(U)$ are equal.
	\end{enumerate}
\end{proposition}

\begin{proof}
	Let $v_1, \dots, v_N$ be a fan compatible with the considered action, and denote by $\theta_{i,j}$ the value of the (non-oriented) angle between $v_i$ and $v_j$. If $D$ is a domain whose boundary is made of the loops $L_{i_1}, \dots, L_{i_k}$ in this cyclic order, the corners of $D$ can be associated to the angles $\theta_{i_1, i_2}, \theta_{i_2, i_3}, \dots, \theta_{i_k, i_1}$, and the sum
	\[ \theta(D) = \theta_{i_1, i_2} + \theta_{i_2, i_3} + \cdots + \theta_{i_{k-1}, i_k} \]
	of these angles is equal to $2\pi$. For a subset $U \subset \sphere{2}$, denote by $S_B(U)$ (res. $S_W(U)$) the sum of the $\theta(D)$ on every black domain (res. white domain) $D$ included in $U$. By definition, $S_B(U) = 2\pi F_B(U)$ and $S_W(U) = 2\pi F_W(U)$.
	
	Remark that $S_B(\sphere{2})$ and $S_W(\sphere{2})$ have exactly the same terms. Indeed, each intersection between two loops $L_i$ and $L_j$ creates four corners with associated angles $\theta_{i,j} = \theta_{j,i}$. Two of these corners lie in black domains while the two others lie in white domains. It follows that $S_B(\sphere{2}) = S_W(\sphere{2})$ and then $F_B(\sphere{2}) = F_W(\sphere{2})$.
	
	Similarly, if $U$ is a connected component of $\sphere{2} \setminus L_i$, $F_B(U) = F_W(U)$ because $S_B(U)$ and $S_W(U)$ have exactly the same terms too. Indeed, as before, a vertex in the interior of $U$ gives four terms: two in $S_B(U)$ and two $S_W(U)$. A vertex on $\boundary{U} = L_i$ creates four corners. Only two of them lie in $U$, and they have different colors (see Figure~\ref{f:circle}). It provides two terms, one in $S_B(U)$ and one in $S_W(U)$.
\end{proof}

\begin{figure}
	\centering
	\def\svgwidth{0.6\textwidth}
	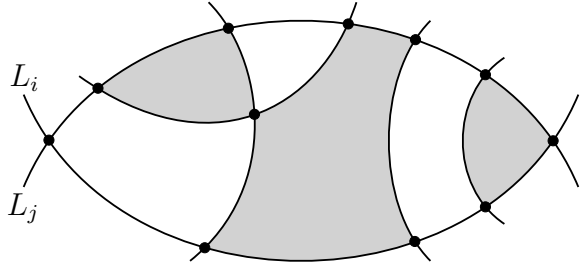
	\caption{An eye between two loops}
	\label{f:eye}
\end{figure}

If two loops $L_i$ and $L_j$ intersect, we call \emph{eye} between $L_i$ and $L_j$ a connected component $U$ of $\sphere{2} \setminus (L_i \cup L_j)$. The two vertices in $L_i \cap L_j \cap \closure{U}$ will be called the corners of the eyes, and any other vertex on the boundary of $U$ is said to be an \emph{eyelash}. Figure~\ref{f:eye} shows an example of an eye between two loops $L_i$ and $L_j$.

\begin{proposition}
	\label{p:black_and_white_domains_inside_eye}
	\label{p:odd_number_of_eyelashes}
	Let $U$ be an eye between two loops $L_i$ and $L_j$. Then:
	\begin{enumerate}
		\item the numbers $F_B(U)$ and $F_W(U)$ are equal,
		\item the two domains on either corners of the eye $U$ have different colors,
		\item the number of eyelashes of $U$ on $L_i$ (res. on $L_j$) is odd.
	\end{enumerate}
\end{proposition}

\begin{proof}
	The proof is the same as for Proposition~\ref{p:black_and_white_domains_on_whole_sphere}. Each intersection point in $\closure{U}$ provides the same terms in $S_B(U)$ and $S_W(U)$, except the two corners of the eye. If the two corner domains of $U$ had the same color, then one would have
	\[ 2 \theta_{i,j} = \abs{S_B(U) - S_W(U)} = 2\pi \abs{F_B(U) - F_W(U)}, \]
	which contradicts the fact that $\theta_{i,j}$ lies in $(0, \pi)$ (indeed, the cones in the fan associated to a hyperbolic domain must be strongly convex). It follows that the two corners have different colors. This implies $F_B(U) = F_W(U)$ and that the there is an odd number of eyelashes on each loop.
\end{proof}

We saw in Theorem~\ref{t:number_of_domains_dimension_n} that the number of domains $F$ of the action is a multiple of $4$. By Proposition~\ref{p:graph_is_4_valent}, it follows that the number of fixed points $V$ is equal to 2 modulo 4. If it is difficult to link directly the number $N$ of loops defined by the one-dimensional orbits of the action to the other constants of the graph, it is still possible to obtain some relations between the modulo classes of these numbers.

First, let us consider the number of fixed points along the closure $L_i$ of some one-dimensional orbit:

\begin{proposition}
	\label{p:vertices_on_loop_divisble_by_4}
	\label{p:vertices_inside_circle_is_odd}
	Consider a loop $L_i$. The following properties hold:
	\begin{enumerate}
		\item the number of vertices on $L_i$ is a multiple of $4$,
		\item the number of vertices in a connected component $U$ of $\sphere{2} \setminus L_i$ is odd.
	\end{enumerate}
\end{proposition}

\begin{proof}
	For topological reasons, any other loop $L_j$ has to cross $L_i$ an even number of times, so the number of intersection points on the loop $L_i$ is even. Then it is possible to associate to each vertex on $L_i$ a sign $+$ or $-$ such that two consecutive vertices along $L_i$ have opposite signs (as in Figure~\ref{f:circle}). The number of vertices on the loop is twice the number of vertices with positive sign. But any loop $L_j$ crossing $L_i$ form an eye between them. By Proposition~\ref{p:odd_number_of_eyelashes}, this eye has an odd number of eyelashes on $L_j$, so its corners have same sign. It follows that vertices with positive sign come in pairs, and then the number of vertices on $L_i$ is a multiple of $4$.
	
	For a subset $U \subset \sphere{2}$, denote by $V(U)$ (res. $E(U)$, $F(U)$) the number of vertices (res. edges, faces) of $\Gamma$ that lie in $U$. If $U$ is a connected component of $\sphere{2} \setminus L_i$, the restriction $\Gamma'$ of the embedded graph $\Gamma$ to $\closure{U}$ is also a planar graph with $V(\closure{U}) = V(U) + V(L_i)$ vertices, $E(\closure{U})$ edges and $F(U) + 1$ faces. In particular those numbers have to satisfy the Euler relation
	\[ (V(U) + V(L_i)) - E(\closure{U}) + (F(U) + 1) = 2. \]
	Now remark that the vertices of $\Gamma'$ in $U$ are $4$-valent, and the ones on $L_i$ are $3$-valent (see Figure~\ref{f:circle}). In similar way as in the proof of Proposition~\ref{p:graph_is_4_valent}, one can prove the relation
	\[ 4V(U) + 3V(L_i) = 2 E(\closure{U}). \]
	Since we showed that $V(L_i)$ is a multiple of $4$, it follows that $E(\closure{U})$ is even. But Proposition~\ref{p:black_and_white_domains_inside_circle} implies that $F(U)$ is even, so finally
	\[ V(U) = 1 - V(L_i) + E(\closure{U}) - F(U) \]
	is odd.
\end{proof}

Now we are able to prove the following result involving the number $N$ of loops defined by the totally hyperbolic action of $\realset^2$ on the sphere.

\begin{theorem}
	\label{t:N_V_F_equivalence}
	Let $\rho : \realset^2 \times \sphere{2} \rightarrow \sphere{2}$ be a totally hyperbolic action of $\realset^2$ on the 2-dimensional sphere with $V$ fixed points and $F$ hyperbolic domains. Denote by $N$ the number of loops on $\sphere{2}$ defined by the 1-dimensional orbits of $\rho$.
	Then the following equivalent assertions hold:
	\begin{enumerate}
		\item $N$ is odd,
		\item $V$ is equal to $2$ modulo $4$,
		\item $F$ is a multiple of $4$.
	\end{enumerate}
\end{theorem}

\begin{proof}
	\begin{figure}
		\centering
		\def\svgwidth{0.8\textwidth}
		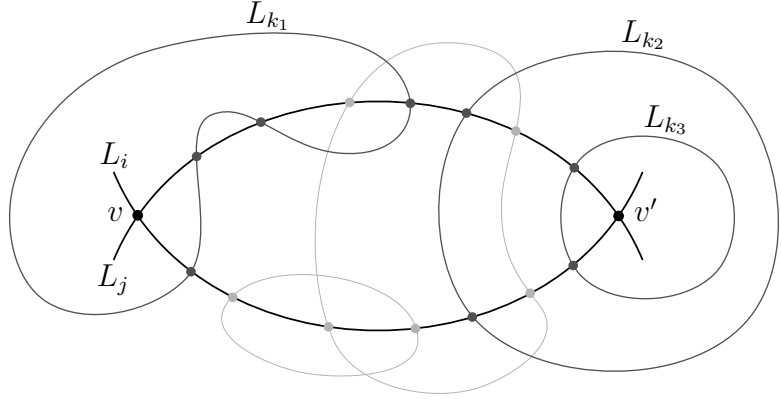
		\caption{Three loops $L_{k_1}, L_{k_2}, L_{k_3}$ separating a pair $p=\pair{v}{v'}$.}
		\label{f:loops_separating_pair}
	\end{figure}
	
	The third assertion follows from Theorem~\ref{t:number_of_domains_dimension_n}. It is equivalent to assertion (2) by the identity $F = V + 2$ (Proposition~\ref{p:graph_is_4_valent}). It suffices now to prove that these two assertions are also equivalent to (1). To achieve this, we will consider a partition of the set of vertices into a set of pairs $P = \set{p_1, \dots, p_{V/2}}$ satisfying the following condition: in any pair $p_j = \pair{v}{v'}$, the vertices $v$ and $v'$ are the corners of some eye between two loops. For $1 \leq i \leq N$, fix $U_i$ a connected component of $\sphere{2} \setminus L_i$, and consider the subset of pairs:
	\[ P_i = \set{\pair{v}{v'} \in P \text{ pair} \mid \set{v, v'} \cap U_i \text{ contains exactly one element} }. \]
	First, remark the following. A vertex $v$ in $U_i$ is either a member of a pair in $P_i$, or a member of a pair $\pair{v}{v'}$ with $v' \in U_i$, the two cases being mutually exclusive. It follows that $U_i \setminus P_i$ contains an even number of vertices. By Proposition~\ref{p:vertices_inside_circle_is_odd}, $U_i$ contains an odd number of vertices, and then $P_i$ contains an odd number of pairs. Then, remark also the following ``dual'' result. Consider a pair $p = \pair{v}{v'}$ in $P$ and the associated eye $V$ between two loops $L_i$ and $L_j$. For any other loop $L_k$, either $v$ and $v'$ belong in distinct connected components of $\sphere{2} \setminus L_k$ (i.e. $p \in P_k$) and then $L_k$ creates an odd number of eyelashes on both $L_i$ and $L_j$, or $v$ and $v'$ are in a same connected component of $L_k$ (i.e. $p \notin P_k$) and then $L_k$ create an even number of eyelashes on both $L_i$ and $L_j$ (see Figure~\ref{f:loops_separating_pair}). Since, according to Proposition~\ref{p:odd_number_of_eyelashes}, the number of eyelashes is odd on both $L_i$ and $L_j$, the set
	\[ L(p) = \set{ L_k \mid p \in P_k} \]
	of loops ``separating'' the pair $ p = \pair{v}{v'}$ has odd cardinality.
	
	These two remarks are enough to prove the equivalence between the above assertions. Indeed, define a $N \times (V/2)$ matrix $(a_{i,j})$ by 
	\[ a_{i,j} = 
	\begin{cases}
	1 &\text{if } p_j \in L_i, \\
	0 &\text{if } p_j \notin L_i.
	\end{cases} \]
	We then have the relation 
	\[ \sum_{i=1}^{N} \card(P_i) = \sum_{i=1}^{N} \sum_{j=1}^{V/2} a_{i,j} = \sum_{j=1}^{V/2} \sum_{i=1}^{N} a_{i,j} = \sum_{j=1}^{V/2} \card L(p_j). \]
	Calling $S$ this sum and using that the $\card(P_i)$ and $\card(L(p_j))$ are all odd, one concludes that:
	\[ N \text{ is odd} \iff S \text{ is odd} \iff \frac{V}{2} \text{ is odd} \iff V = 2\mod{4}. \]
\end{proof}

We can show that Condition (1) on $N$ in the above theorem is optimal, in the sense that we can construct examples of totally hyperbolic actions with $N$ loops for any odd number $N$.

\begin{proposition}
	Let $N$ be an odd number, $N \geq 3$. There exists a totally hyperbolic action $\rho$ of $\realset^2$ on the sphere $\sphere{2}$ such that the closure of the one-dimensional orbits of $\rho$ define exactly $N$ loops.
\end{proposition}

\begin{proof}
	It is exactly the same proof as for Corollary~\ref{c:augment_domains}. Starting from the classical example of the hyperbolic action of $\realset^2$ on $\sphere{2}$ with 3 loops, iterate the construction defined in Proposition~\ref{p:adding_spheres_around_fixed_point}, which adds exactly two loops to the action.
\end{proof}

\section{Examples of totally hyperbolic actions in dimension 3}
\label{s:in_dimension_3}

In this section, we provide constructive examples of totally hyperbolic actions on two 3-dimensional manifolds: the sphere $\sphere{3}$ and the projective space $\projspace{3}$.

\subsection{On the sphere $\sphere{3}$} Define the 3-dimensional sphere as the subspace
\[ \sphere{3} = \set{ (z_1, z_2) \in \complexset^2 \mid \abs{z_1}^2 + \abs{z_2}^2 = 1 } \subset \complexset^2. \]
Let $f : \sphere{3} \rightarrow \realset$ be the map defined by $f(z_1, z_2) = \abs{z_2}^2 - \abs{z_1}^2$. The zero level set of $f$ is a 2-dimensional torus $\sphere{1} \times \sphere{1}$ embedded in $\sphere{3}$, while the sets $A_1 = \set{f \geq 0}$ and $A_2 = \set{f \leq 0}$ are two solid tori $\closeddisk \times \sphere{1}$ embedded in $\sphere{3}$. This is the usual decomposition of $\sphere{3}$ into two solid tori glued along their common boundary, in such a way that a meridian on the first solid tori $A_1$ is identified with a longitude on the second solid tori $A_2$, and vice versa.

For $k \in \set{1,2}$ and a given slope $\lambda = \slope{a}{b} \in \projspace{1}$, consider the hypersurface $H^k_{\lambda}$ in $\sphere{3}$ consisting of the elements $(z_1, z_2) \in \sphere{3}$ such that $z_k = x_k + i y_k$ is on the real line of slope $\slope{a}{b}$ in $\complexset$. That is, $H^k_{\lambda}$ is the hypersurface defined by the zero level-set of the submersion $(z_1, z_2) \mapsto bx_k - ay_k$. Under the natural identification between the solid tori $A_1, A_2$ and the solid cylinders $C_1, C_2 = \closeddisk \times [0,1]$ in $\realset^3$, remark that: 
\begin{itemize}
	\item $H^k_\lambda \cap A_k$ corresponds to two horizontal closed disks $\closeddisk \times \set{t_1, t_2}$ in $C_k$,
	\item $H^k_\lambda \cap A_{\bar{k}}$ corresponds to a vertical plane $\set{yb - ax = 0} \times [0,1]$ in $C_{\bar{k}}$,
\end{itemize}
where $\{k, \bar{k}\} = \set{1,2}$ (see Figure~\ref{f:domains_S3}). Any two such hypersurfaces $H^k_\lambda$ and $H^l_\mu$ intersect transversally with each other as long as $k \neq l$ or $\lambda \neq \mu$. Moreover, any such surface intersects transversally with the torus $A_1 \cap A_2$, and does not intersect itself. Hence this family of hypersurfaces will be useful to construct decompositions of $\sphere{3}$ into hyperbolic domains of some totally hyperbolic action of $\realset^3$ on $\sphere{3}$.

\begin{figure}
	\centering
	\def\svgwidth{0.8\textwidth}
	\small%
	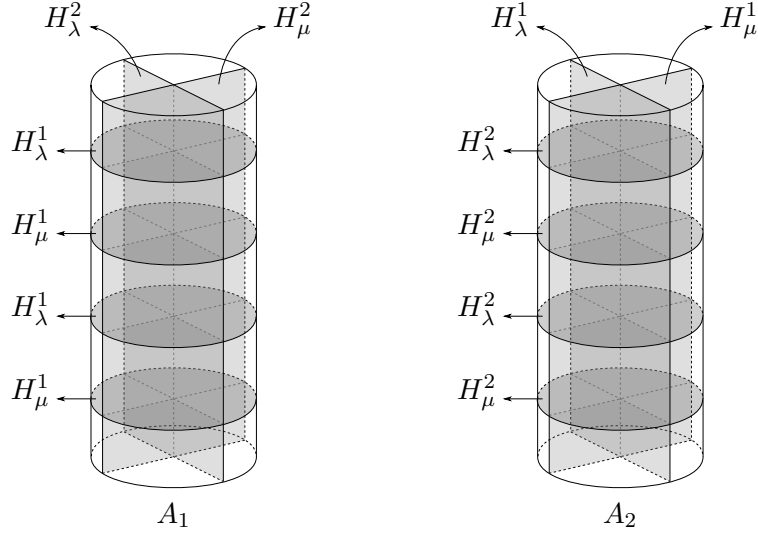
	\caption{A decomposition of $S^3 = A_1 \cup A_2$ into hyperbolic domains}
	\label{f:domains_S3}
\end{figure}

\begin{proposition}
	\label{p:existence_action_S3}
	Let $\lambda, \mu \in \projspace{1}$ be two distinct slopes. Then there exists a totally hyperbolic action $\rho$ of $\realset^3$ on the 3-dimensional sphere $\sphere{3}$ such that the hyperbolic domains of $\rho$ are given by the splitting of $\sphere{3}$ along the torus $H_0 = A_1 \cap A_2$ and the hypersurfaces $H^1_\lambda$, $H^1_\mu$, $H^2_\lambda$ and $H^2_\mu$.
\end{proposition}

\begin{proof}
	It suffices to exhibit a complete fan in $\realset^3$ compatible with the given decomposition of $\sphere{3}$. Note that, on a combinatoric point of view, all the domains in $A_1$ are equivalent, as well as all the domains in $A_2$ (see Figure~\ref{f:domains_S3}). Moreover, the whole system is symmetric with respect to the permutation of the indices 1 and 2. Consider the following fan:
	\[ v_0 = (-1, 0, 0), \begin{matrix}
 & v^1_\lambda = (1, -1, 0), & v^1_\mu = (1, 1, 0), \\
	& v^2_\lambda = (1, 0, -1), & v^2_\mu = (1, 0, 1).
	\end{matrix} \]
	Since it is compatible with any domain in $A_1$ and is symmetric (up to rotation of the fan) when one permutes $v^1_\cdot$ and $v^2_\cdot$, it is compatible with the whole decomposition of $\sphere{3}$.
\end{proof}

\subsection{On the projective space $\projspace{3}$} The construction above extends immediately to the 3-dimensional projective space $\projspace{3}$. Recall that the later is defined as the quotient of the 3-dimensional sphere $\sphere{3}$ by the free involution $\sigma : (z_1, z_2) \mapsto (- z_1, - z_2)$. The map $f : \sphere{3} \rightarrow \realset$ defined above satisfies $f \circ \sigma = f$, hence the solid tori $A_1$, $A_2$ and their common boundary $A_1 \cap A_2$ are invariant by the action of $\sigma$ on $\sphere{3}$. For any $i \in \set{1,2}$ and $\lambda \in \projspace{1}$, the hypersurface $H^i_\lambda \subset \realset^3$ is also invariant by the action of $\sigma$, and thus projects on a well-defined hypersurface in $\projspace{3}$ by the quotient map $\pi : \sphere{3} \rightarrow \projspace{3}$.

\begin{proposition}
	\label{p:existence_action_RP3}
	Le $\lambda, \mu \in \projspace{1}$ be two distinct slopes. Then there exists a totally hyperbolic action $\rho$ of $\realset^3$ on the 3-dimensional projective space $\projspace{3}$ such that the hyperbolic domains of $\rho$ are given by the splitting of $\projspace{3}$ along the hypersurfaces $\pi(A_1 \cap A_2)$, $\pi(H^1_\lambda)$, $\pi(H^1_\mu)$, $\pi(H^2_\lambda)$ and $\pi(H^2_\mu)$.
\end{proposition}

\begin{proof}
	It suffices to remark that, in Proposition~\ref{p:existence_action_S3}, the totally hyperbolic action $\rho$ defined on $\sphere{3}$ can be constructed in such a way that it is symmetric with respect to $\sigma$, by defining it first on any of the domains, and then extending it via the reflection principle. By taking the quotient with respect to $\sigma$, one obtains a totally hyperbolic action on $\projspace{3}$.
\end{proof}

In~\cite{zung2014geometry}, Minh and Zung raise the question of the existence of totally hyperbolic actions on lens spaces. Given two co-prime integers $p$ and $q$, the lens space $\lensspace{p}{q}$ is obtained as the quotient manifold of the 3-dimensional sphere $\sphere{3}$ by the proper and free action of $\modintegerset{p}$ defined by: 
\[ \modclass{1} \cdot (z_1, z_2) = (z_1 \xi, z_2 \xi^q), \text{ where } \xi = \exp({2i\pi}/{p}). \]
In particular, when $(p,q) = (2,1)$, one obtains the 3-dimensional projective space $\projspace{3}$ \cite{hempel1976three}. So Proposition~\ref{p:existence_action_RP3} answers the question in this very specific case. Unfortunately, its constructive proof might not be used for a wider class of lens spaces, since the hypersurfaces $H^i_\lambda$ introduced above are not invariant by the action of $\modintegerset{p}$ on $\sphere{3}$ when $p \geq 3$. Moreover, any $\modintegerset{p}$-invariant hypersurface containing $H^i_\lambda$ contains also $\modclass{1} \cdot H^i_\lambda = H^i_\mu$ with $\mu \neq \lambda$. Therefore it self-intersects and cannot define the boundaries of hyperbolic domains of some hyperbolic action of $\realset^3$ on $\lensspace{p}{q}$.

\bibliography{biblio}

\end{document}

%% file: flow_domain.pdf_tex
%% Creator: Inkscape inkscape 0.48.4, www.inkscape.org
%% Accompanies image file 'flow_domain.pdf' (pdf, eps, ps)
\begingroup%
  \makeatletter%
  \providecommand\color[2][]{%
    \errmessage{(Inkscape) Color is used for the text in Inkscape, but the package 'color.sty' is not loaded}%
    \renewcommand\color[2][]{}%
  }%
  \providecommand\transparent[1]{%
    \errmessage{(Inkscape) Transparency is used (non-zero) for the text in Inkscape, but the package 'transparent.sty' is not loaded}%
    \renewcommand\transparent[1]{}%
  }%
  \providecommand\rotatebox[2]{#2}%
  \ifx\svgwidth\undefined%
    \setlength{\unitlength}{930.16933594bp}%
    \ifx\svgscale\undefined%
      \relax%
    \else%
      \setlength{\unitlength}{\unitlength * \real{\svgscale}}%
    \fi%
  \else%
    \setlength{\unitlength}{\svgwidth}%
  \fi%
  \global\let\svgwidth\undefined%
  \global\let\svgscale\undefined%
  \makeatother%
  \begin{picture}(1,0.44895051)%
    \put(0,0){\includegraphics[width=\unitlength]{flow_domain.pdf}}%
    \put(0.42104807,0.24167642){\color[rgb]{0,0,0}\makebox(0,0)[b]{\smash{$p_+$}}}%
    \put(0.08885296,0.33378268){\color[rgb]{0,0,0}\makebox(0,0)[b]{\smash{$p_-$}}}%
    \put(0.21158898,0.33200104){\color[rgb]{0,0,0}\makebox(0,0)[b]{\smash{$L_1$}}}%
    \put(0.37801239,0.3227531){\color[rgb]{0,0,0}\makebox(0,0)[lb]{\smash{$L_2$}}}%
    \put(0.33217963,0.10635134){\color[rgb]{0,0,0}\makebox(0,0)[lb]{\smash{$L_3$}}}%
    \put(0.18107066,0.07344712){\color[rgb]{0,0,0}\makebox(0,0)[b]{\smash{$L_4$}}}%
    \put(0.08400689,0.21085306){\color[rgb]{0,0,0}\makebox(0,0)[rb]{\smash{$L_5$}}}%
    \put(0.89812539,0.29130858){\color[rgb]{0,0,0}\makebox(0,0)[lb]{\smash{$w$}}}%
    \put(0.60873094,0.1406889){\color[rgb]{0,0,0}\makebox(0,0)[rb]{\smash{$-w$}}}%
    \put(0.78685039,0.36767961){\color[rgb]{0,0,0}\makebox(0,0)[b]{\smash{$v_2$}}}%
    \put(0.91539679,0.22842415){\color[rgb]{0,0,0}\makebox(0,0)[lb]{\smash{$v_3$}}}%
    \put(0.84418759,0.1229976){\color[rgb]{0,0,0}\makebox(0,0)[lb]{\smash{$v_4$}}}%
    \put(0.69067188,0.07093172){\color[rgb]{0,0,0}\makebox(0,0)[b]{\smash{$v_5$}}}%
    \put(0.60744033,0.27003987){\color[rgb]{0,0,0}\makebox(0,0)[rb]{\smash{$v_1$}}}%
  \end{picture}%
\endgroup%

%% file: Morse_construction_step2.pdf_tex
%% Creator: Inkscape inkscape 0.48.4, www.inkscape.org
%% Accompanies image file 'Morse_construction_step2.pdf' (pdf, eps, ps)
\begingroup%
  \makeatletter%
  \providecommand\color[2][]{%
    \errmessage{(Inkscape) Color is used for the text in Inkscape, but the package 'color.sty' is not loaded}%
    \renewcommand\color[2][]{}%
  }%
  \providecommand\transparent[1]{%
    \errmessage{(Inkscape) Transparency is used (non-zero) for the text in Inkscape, but the package 'transparent.sty' is not loaded}%
    \renewcommand\transparent[1]{}%
  }%
  \providecommand\rotatebox[2]{#2}%
  \ifx\svgwidth\undefined%
    \setlength{\unitlength}{1073.6bp}%
    \ifx\svgscale\undefined%
      \relax%
    \else%
      \setlength{\unitlength}{\unitlength * \real{\svgscale}}%
    \fi%
  \else%
    \setlength{\unitlength}{\svgwidth}%
  \fi%
  \global\let\svgwidth\undefined%
  \global\let\svgscale\undefined%
  \makeatother%
  \begin{picture}(1,0.45948894)%
    \put(0,0){\includegraphics[width=\unitlength]{Morse_construction_step2.pdf}}%
    \put(0.25127128,0.2625386){\color[rgb]{0,0,0}\makebox(0,0)[rb]{\smash{$q_i$}}}%
    \put(0.73652224,0.2465709){\color[rgb]{0,0,0}\makebox(0,0)[lb]{\smash{$q_j$}}}%
    \put(0.84246835,0.24329665){\color[rgb]{0,0,0}\makebox(0,0)[lb]{\smash{$p_j$}}}%
    \put(0.16729363,0.25751166){\color[rgb]{0,0,0}\makebox(0,0)[rb]{\smash{$p_i$}}}%
    \put(0.12589033,0.33971541){\color[rgb]{0,0,0}\makebox(0,0)[b]{\smash{$U_i$}}}%
    \put(0.89292833,0.34415092){\color[rgb]{0,0,0}\makebox(0,0)[b]{\smash{$U_j$}}}%
    \put(0.4778889,0.25278029){\color[rgb]{0,0,0}\makebox(0,0)[b]{\smash{$\Ocal_{ij}$}}}%
    \put(0.58399347,0.1580963){\color[rgb]{0,0,0}\makebox(0,0)[lb]{\smash{$V_{ij}$}}}%
  \end{picture}%
\endgroup%

%% file: Morse_construction_step3.pdf_tex
%% Creator: Inkscape inkscape 0.48.3.1, www.inkscape.org
%% Accompanies image file 'Morse_construction_step3.pdf' (pdf, eps, ps)
\begingroup%
  \makeatletter%
  \providecommand\color[2][]{%
    \errmessage{(Inkscape) Color is used for the text in Inkscape, but the package 'color.sty' is not loaded}%
    \renewcommand\color[2][]{}%
  }%
  \providecommand\transparent[1]{%
    \errmessage{(Inkscape) Transparency is used (non-zero) for the text in Inkscape, but the package 'transparent.sty' is not loaded}%
    \renewcommand\transparent[1]{}%
  }%
  \providecommand\rotatebox[2]{#2}%
  \ifx\svgwidth\undefined%
    \setlength{\unitlength}{1565.421875bp}%
    \ifx\svgscale\undefined%
      \relax%
    \else%
      \setlength{\unitlength}{\unitlength * \real{\svgscale}}%
    \fi%
  \else%
    \setlength{\unitlength}{\svgwidth}%
  \fi%
  \global\let\svgwidth\undefined%
  \global\let\svgscale\undefined%
  \makeatother%
  \begin{picture}(1,0.50422713)%
    \put(0,0){\includegraphics[width=\unitlength]{Morse_construction_step3.pdf}}%
    \put(0.23559144,0.28907744){\color[rgb]{0,0,0}\makebox(0,0)[lb]{\smash{$p_i$}}}%
    \put(0.74663579,0.28907744){\color[rgb]{0,0,0}\makebox(0,0)[lb]{\smash{$p_i$}}}%
    \put(0.48089273,0.27885655){\color[rgb]{0,0,0}\makebox(0,0)[b]{\smash{\mbox{\small or}}}}%
    \put(0.26971114,0.36362978){\color[rgb]{0,0,0}\makebox(0,0)[lb]{\smash{$\{ f_i = d_i^\pm \}$}}}%
    \put(0.83862377,0.41172808){\color[rgb]{0,0,0}\makebox(0,0)[lb]{\smash{$\{ f_i = d_i^- \}$}}}%
    \put(0.86928643,0.37084454){\color[rgb]{0,0,0}\makebox(0,0)[lb]{\smash{$\{ f_i = d_i^+ \}$}}}%
    \put(0.15699798,0.38223224){\color[rgb]{0,0,0}\makebox(0,0)[b]{\smash{$U_i$}}}%
    \put(0.1920878,0.31390298){\color[rgb]{0,0,0}\makebox(0,0)[b]{\smash{$W_i$}}}%
    \put(0.40185048,0.26310969){\color[rgb]{0,0,0}\makebox(0,0)[b]{\smash{$W_{ij}$}}}%
    \put(0.69420908,0.32039244){\color[rgb]{0,0,0}\makebox(0,0)[b]{\smash{$W_i$}}}%
    \put(0.91370599,0.26380972){\color[rgb]{0,0,0}\makebox(0,0)[b]{\smash{$W_{ij}$}}}%
    \put(0.69026296,0.39260269){\color[rgb]{0,0,0}\makebox(0,0)[b]{\smash{$U_i$}}}%
    \put(0.22537056,0.01311349){\color[rgb]{0,0,0}\makebox(0,0)[b]{\smash{\mbox{\small $\Ind_{p_i}(w) = 0 \text{ or } 2$}}}}%
    \put(0.7364149,0.02333438){\color[rgb]{0,0,0}\makebox(0,0)[b]{\smash{\mbox{\small $\Ind_{p_i}(w) = 1$}}}}%
  \end{picture}%
\endgroup%

%% file: jigsaw_puzzles.pdf_tex
%% Creator: Inkscape inkscape 0.48.4, www.inkscape.org
%% Accompanies image file 'jigsaw_puzzles.pdf' (pdf, eps, ps)
\begingroup%
  \makeatletter%
  \providecommand\color[2][]{%
    \errmessage{(Inkscape) Color is used for the text in Inkscape, but the package 'color.sty' is not loaded}%
    \renewcommand\color[2][]{}%
  }%
  \providecommand\transparent[1]{%
    \errmessage{(Inkscape) Transparency is used (non-zero) for the text in Inkscape, but the package 'transparent.sty' is not loaded}%
    \renewcommand\transparent[1]{}%
  }%
  \providecommand\rotatebox[2]{#2}%
  \ifx\svgwidth\undefined%
    \setlength{\unitlength}{2736.678125bp}%
    \ifx\svgscale\undefined%
      \relax%
    \else%
      \setlength{\unitlength}{\unitlength * \real{\svgscale}}%
    \fi%
  \else%
    \setlength{\unitlength}{\svgwidth}%
  \fi%
  \global\let\svgwidth\undefined%
  \global\let\svgscale\undefined%
  \makeatother%
  \begin{picture}(1,0.56220989)%
    \put(0,0){\includegraphics[width=\unitlength]{jigsaw_puzzles.pdf}}%
    \put(0.72755292,0.54810976){\color[rgb]{0,0,0}\makebox(0,0)[b]{\smash{$v_1$}}}%
    \put(0.82109699,0.42240992){\color[rgb]{0,0,0}\makebox(0,0)[lb]{\smash{$v_2$}}}%
    \put(0.78601796,0.360437){\color[rgb]{0,0,0}\makebox(0,0)[b]{\smash{$v_3$}}}%
    \put(0.66908788,0.360437){\color[rgb]{0,0,0}\makebox(0,0)[b]{\smash{$v_4$}}}%
    \put(0.63400885,0.42240992){\color[rgb]{0,0,0}\makebox(0,0)[rb]{\smash{$v_5$}}}%
    \put(0.76847845,0.53524745){\color[rgb]{0,0,0}\makebox(0,0)[b]{\smash{$w_A$}}}%
    \put(0.64570186,0.39610065){\color[rgb]{0,0,0}\makebox(0,0)[rb]{\smash{$w_B$}}}%
    \put(0.22826147,0.50367637){\color[rgb]{0,0,0}\makebox(0,0)[b]{\smash{$H_3$}}}%
    \put(0.26334049,0.44228805){\color[rgb]{0,0,0}\makebox(0,0)[lb]{\smash{$H_1$}}}%
    \put(0.11425463,0.36336025){\color[rgb]{0,0,0}\makebox(0,0)[rb]{\smash{$H_5$}}}%
    \put(0.15050296,0.32828122){\color[rgb]{0,0,0}\makebox(0,0)[rb]{\smash{$H_2$}}}%
    \put(0.04409659,0.27566269){\color[rgb]{0,0,0}\makebox(0,0)[rb]{\smash{$H_4$}}}%
    \put(0.57875939,0.31688053){\color[rgb]{0,0,0}\makebox(0,0)[b]{\smash{{\normalsize A}}}}%
    \put(0.8710846,0.31688054){\color[rgb]{0,0,0}\makebox(0,0)[b]{\smash{{\normalsize B}}}}%
  \end{picture}%
\endgroup%

%% file: spheres_around_fixed_point.pdf_tex
%% Creator: Inkscape inkscape 0.48.4, www.inkscape.org
%% Accompanies image file 'spheres_around_fixed_point.pdf' (pdf, eps, ps)
\begingroup%
  \makeatletter%
  \providecommand\color[2][]{%
    \errmessage{(Inkscape) Color is used for the text in Inkscape, but the package 'color.sty' is not loaded}%
    \renewcommand\color[2][]{}%
  }%
  \providecommand\transparent[1]{%
    \errmessage{(Inkscape) Transparency is used (non-zero) for the text in Inkscape, but the package 'transparent.sty' is not loaded}%
    \renewcommand\transparent[1]{}%
  }%
  \providecommand\rotatebox[2]{#2}%
  \ifx\svgwidth\undefined%
    \setlength{\unitlength}{1326.9328125bp}%
    \ifx\svgscale\undefined%
      \relax%
    \else%
      \setlength{\unitlength}{\unitlength * \real{\svgscale}}%
    \fi%
  \else%
    \setlength{\unitlength}{\svgwidth}%
  \fi%
  \global\let\svgwidth\undefined%
  \global\let\svgscale\undefined%
  \makeatother%
  \begin{picture}(1,0.43178995)%
    \put(0,0){\includegraphics[width=\unitlength]{spheres_around_fixed_point.pdf}}%
    \put(0.80848102,0.39850834){\color[rgb]{0,0,0}\makebox(0,0)[b]{\smash{$H_1$}}}%
    \put(1.00140715,0.20558222){\color[rgb]{0,0,0}\makebox(0,0)[lb]{\smash{$H_2$}}}%
    \put(0.2055869,0.39851302){\color[rgb]{0,0,0}\makebox(0,0)[b]{\smash{$H_1$}}}%
    \put(0.39851302,0.2055869){\color[rgb]{0,0,0}\makebox(0,0)[lb]{\smash{$H_2$}}}%
    \put(0.12118172,0.28999211){\color[rgb]{0,0,0}\makebox(0,0)[b]{\smash{$U_i$}}}%
    \put(0.69996008,0.32616576){\color[rgb]{0,0,0}\makebox(0,0)[b]{\smash{$V_i$}}}%
    \put(0.74819161,0.26587635){\color[rgb]{0,0,0}\makebox(0,0)[b]{\smash{$C_i$}}}%
    \put(0.78798262,0.22608533){\color[rgb]{0,0,0}\makebox(0,0)[b]{\smash{$T_i$}}}%
    \put(0.92905985,0.12118176){\color[rgb]{0,0,0}\makebox(0,0)[b]{\smash{$S'$}}}%
    \put(0.88082832,0.16941329){\color[rgb]{0,0,0}\makebox(0,0)[b]{\smash{$S$}}}%
  \end{picture}%
\endgroup%

%% file: sixteen_domains.pdf_tex
%% Creator: Inkscape inkscape 0.48.3.1, www.inkscape.org
%% Accompanies image file 'sixteen_domains.pdf' (pdf, eps, ps)
\begingroup%
  \makeatletter%
  \providecommand\color[2][]{%
    \errmessage{(Inkscape) Color is used for the text in Inkscape, but the package 'color.sty' is not loaded}%
    \renewcommand\color[2][]{}%
  }%
  \providecommand\transparent[1]{%
    \errmessage{(Inkscape) Transparency is used (non-zero) for the text in Inkscape, but the package 'transparent.sty' is not loaded}%
    \renewcommand\transparent[1]{}%
  }%
  \providecommand\rotatebox[2]{#2}%
  \ifx\svgwidth\undefined%
    \setlength{\unitlength}{929.6bp}%
    \ifx\svgscale\undefined%
      \relax%
    \else%
      \setlength{\unitlength}{\unitlength * \real{\svgscale}}%
    \fi%
  \else%
    \setlength{\unitlength}{\svgwidth}%
  \fi%
  \global\let\svgwidth\undefined%
  \global\let\svgscale\undefined%
  \makeatother%
  \begin{picture}(1,0.5524957)%
    \put(0,0){\includegraphics[width=\unitlength]{sixteen_domains.pdf}}%
    \put(0.13855422,0.5){\color[rgb]{0,0,0}\makebox(0,0)[b]{\smash{$P_1$}}}%
    \put(0.86144578,0.5){\color[rgb]{0,0,0}\makebox(0,0)[b]{\smash{$P_2$}}}%
  \end{picture}%
\endgroup%

%% file: circle.pdf_tex
%% Creator: Inkscape inkscape 0.48.3.1, www.inkscape.org
%% Accompanies image file 'circle.pdf' (pdf, eps, ps)
\begingroup%
  \makeatletter%
  \providecommand\color[2][]{%
    \errmessage{(Inkscape) Color is used for the text in Inkscape, but the package 'color.sty' is not loaded}%
    \renewcommand\color[2][]{}%
  }%
  \providecommand\transparent[1]{%
    \errmessage{(Inkscape) Transparency is used (non-zero) for the text in Inkscape, but the package 'transparent.sty' is not loaded}%
    \renewcommand\transparent[1]{}%
  }%
  \providecommand\rotatebox[2]{#2}%
  \ifx\svgwidth\undefined%
    \setlength{\unitlength}{545.6bp}%
    \ifx\svgscale\undefined%
      \relax%
    \else%
      \setlength{\unitlength}{\unitlength * \real{\svgscale}}%
    \fi%
  \else%
    \setlength{\unitlength}{\svgwidth}%
  \fi%
  \global\let\svgwidth\undefined%
  \global\let\svgscale\undefined%
  \makeatother%
  \begin{picture}(1,1.02527034)%
    \put(0,0){\includegraphics[width=\unitlength]{circle.pdf}}%
    \put(0.29434535,0.91356491){\color[rgb]{0,0,0}\makebox(0,0)[b]{\smash{$L_i$}}}%
    \put(0.93737757,0.65933326){\color[rgb]{0,0,0}\makebox(0,0)[b]{\smash{{\tiny $+$}}}}%
    \put(0.67379716,0.05904033){\color[rgb]{0,0,0}\makebox(0,0)[b]{\smash{{\tiny $+$}}}}%
    \put(0.12568938,0.21265025){\color[rgb]{0,0,0}\makebox(0,0)[b]{\smash{{\tiny $+$}}}}%
    \put(0.12568938,0.77646814){\color[rgb]{0,0,0}\makebox(0,0)[b]{\smash{{\tiny $+$}}}}%
    \put(0.6301581,0.9422969){\color[rgb]{0,0,0}\makebox(0,0)[b]{\smash{{\tiny $-$}}}}%
    \put(0.34737631,0.05031248){\color[rgb]{0,0,0}\makebox(0,0)[b]{\smash{{\tiny $-$}}}}%
    \put(0.82217022,0.15155531){\color[rgb]{0,0,0}\makebox(0,0)[b]{\smash{{\tiny $-$}}}}%
    \put(0.0349199,0.39069794){\color[rgb]{0,0,0}\makebox(0,0)[b]{\smash{{\tiny $-$}}}}%
  \end{picture}%
\endgroup%

%% file: eye.pdf_tex
%% Creator: Inkscape inkscape 0.48.3.1, www.inkscape.org
%% Accompanies image file 'eye.pdf' (pdf, eps, ps)
\begingroup%
  \makeatletter%
  \providecommand\color[2][]{%
    \errmessage{(Inkscape) Color is used for the text in Inkscape, but the package 'color.sty' is not loaded}%
    \renewcommand\color[2][]{}%
  }%
  \providecommand\transparent[1]{%
    \errmessage{(Inkscape) Transparency is used (non-zero) for the text in Inkscape, but the package 'transparent.sty' is not loaded}%
    \renewcommand\transparent[1]{}%
  }%
  \providecommand\rotatebox[2]{#2}%
  \ifx\svgwidth\undefined%
    \setlength{\unitlength}{545.6bp}%
    \ifx\svgscale\undefined%
      \relax%
    \else%
      \setlength{\unitlength}{\unitlength * \real{\svgscale}}%
    \fi%
  \else%
    \setlength{\unitlength}{\svgwidth}%
  \fi%
  \global\let\svgwidth\undefined%
  \global\let\svgscale\undefined%
  \makeatother%
  \begin{picture}(1,0.53079179)%
    \put(0,0){\includegraphics[width=\unitlength]{eye.pdf}}%
    \put(0.06153003,0.33783254){\color[rgb]{0,0,0}\makebox(0,0)[b]{\smash{$L_i$}}}%
    \put(0.05933409,0.13345829){\color[rgb]{0,0,0}\makebox(0,0)[b]{\smash{$L_j$}}}%
  \end{picture}%
\endgroup%

%% file: loops_separating_pair.pdf_tex
%% Creator: Inkscape inkscape 0.48.3.1, www.inkscape.org
%% Accompanies image file 'loops_separating_pair.pdf' (pdf, eps, ps)
\begingroup%
  \makeatletter%
  \providecommand\color[2][]{%
    \errmessage{(Inkscape) Color is used for the text in Inkscape, but the package 'color.sty' is not loaded}%
    \renewcommand\color[2][]{}%
  }%
  \providecommand\transparent[1]{%
    \errmessage{(Inkscape) Transparency is used (non-zero) for the text in Inkscape, but the package 'transparent.sty' is not loaded}%
    \renewcommand\transparent[1]{}%
  }%
  \providecommand\rotatebox[2]{#2}%
  \ifx\svgwidth\undefined%
    \setlength{\unitlength}{762.75786133bp}%
    \ifx\svgscale\undefined%
      \relax%
    \else%
      \setlength{\unitlength}{\unitlength * \real{\svgscale}}%
    \fi%
  \else%
    \setlength{\unitlength}{\svgwidth}%
  \fi%
  \global\let\svgwidth\undefined%
  \global\let\svgscale\undefined%
  \makeatother%
  \begin{picture}(1,0.54498706)%
    \put(0,0){\includegraphics[width=\unitlength]{loops_separating_pair.pdf}}%
    \put(0.16750175,0.31609998){\color[rgb]{0,0,0}\makebox(0,0)[b]{\smash{$L_i$}}}%
    \put(0.165931,0.16991126){\color[rgb]{0,0,0}\makebox(0,0)[b]{\smash{$L_j$}}}%
    \put(0.35039046,0.48390936){\color[rgb]{0,0,0}\makebox(0,0)[b]{\smash{$L_{k_1}$}}}%
    \put(0.79678218,0.46115191){\color[rgb]{0,0,0}\makebox(0,0)[b]{\smash{$L_{k_2}$}}}%
    \put(0.82358811,0.3599321){\color[rgb]{0,0,0}\makebox(0,0)[b]{\smash{$L_{k_3}$}}}%
    \put(0.17833215,0.24771115){\color[rgb]{0,0,0}\makebox(0,0)[rb]{\smash{$v$}}}%
    \put(0.78579876,0.24742533){\color[rgb]{0,0,0}\makebox(0,0)[lb]{\smash{$v'$}}}%
  \end{picture}%
\endgroup%

%% file: domains_S3.pdf_tex
%% Creator: Inkscape inkscape 0.48.3.1, www.inkscape.org
%% Accompanies image file 'domains_S3.pdf' (pdf, eps, ps)
\begingroup%
  \makeatletter%
  \providecommand\color[2][]{%
    \errmessage{(Inkscape) Color is used for the text in Inkscape, but the package 'color.sty' is not loaded}%
    \renewcommand\color[2][]{}%
  }%
  \providecommand\transparent[1]{%
    \errmessage{(Inkscape) Transparency is used (non-zero) for the text in Inkscape, but the package 'transparent.sty' is not loaded}%
    \renewcommand\transparent[1]{}%
  }%
  \providecommand\rotatebox[2]{#2}%
  \ifx\svgwidth\undefined%
    \setlength{\unitlength}{1626.240625bp}%
    \ifx\svgscale\undefined%
      \relax%
    \else%
      \setlength{\unitlength}{\unitlength * \real{\svgscale}}%
    \fi%
  \else%
    \setlength{\unitlength}{\svgwidth}%
  \fi%
  \global\let\svgwidth\undefined%
  \global\let\svgscale\undefined%
  \makeatother%
  \begin{picture}(1,0.66110489)%
    \put(0,0){\includegraphics[width=\unitlength]{domains_S3.pdf}}%
    \put(0.15351741,0.61278024){\color[rgb]{0,0,0}\makebox(0,0)[rb]{\smash{$H^2_\lambda$}}}%
    \put(0.37980619,0.61278024){\color[rgb]{0,0,0}\makebox(0,0)[lb]{\smash{$H^2_\mu$}}}%
    \put(0.11416284,0.46520061){\color[rgb]{0,0,0}\makebox(0,0)[rb]{\smash{$H^1_\lambda$}}}%
    \put(0.11416284,0.36681418){\color[rgb]{0,0,0}\makebox(0,0)[rb]{\smash{$H^1_\mu$}}}%
    \put(0.11416284,0.26842776){\color[rgb]{0,0,0}\makebox(0,0)[rb]{\smash{$H^1_\lambda$}}}%
    \put(0.11416284,0.17004134){\color[rgb]{0,0,0}\makebox(0,0)[rb]{\smash{$H^1_\mu$}}}%
    \put(0.26174248,0.0224617){\color[rgb]{0,0,0}\makebox(0,0)[b]{\smash{$A_1$}}}%
    \put(0.6848041,0.61278024){\color[rgb]{0,0,0}\makebox(0,0)[rb]{\smash{$H^1_\lambda$}}}%
    \put(0.91109288,0.61278024){\color[rgb]{0,0,0}\makebox(0,0)[lb]{\smash{$H^1_\mu$}}}%
    \put(0.64544953,0.46520061){\color[rgb]{0,0,0}\makebox(0,0)[rb]{\smash{$H^2_\lambda$}}}%
    \put(0.64544953,0.36681418){\color[rgb]{0,0,0}\makebox(0,0)[rb]{\smash{$H^2_\mu$}}}%
    \put(0.64544953,0.26842776){\color[rgb]{0,0,0}\makebox(0,0)[rb]{\smash{$H^2_\lambda$}}}%
    \put(0.64544953,0.17004134){\color[rgb]{0,0,0}\makebox(0,0)[rb]{\smash{$H^2_\mu$}}}%
    \put(0.79302917,0.0224617){\color[rgb]{0,0,0}\makebox(0,0)[b]{\smash{$A_2$}}}%
  \end{picture}%
\endgroup%

%% file: main.bbl
\begin{bibdiv}
\begin{biblist}

\bib{ayoul2010galoisian}{article}{
      author={Ayoul, Micha{\"e}l},
      author={Zung, Nguyen~Tien},
       title={Galoisian obstructions to non-{H}amiltonian integrability},
        date={2010},
     journal={Comptes Rendus Mathematique},
      volume={348},
      number={23},
       pages={1323\ndash 1326},
}

\bib{babelon2003introduction}{book}{
      author={Babelon, Olivier},
      author={Bernard, Denis},
      author={Talon, Michel},
       title={Introduction to classical integrable systems},
   publisher={Cambridge University Press},
        date={2003},
}

\bib{bogoyavlenskij1998extended}{article}{
      author={Bogoyavlenskij, Oleg~I},
       title={Extended integrability and bi-{H}amiltonian systems},
        date={1998},
     journal={Communications in mathematical physics},
      volume={196},
      number={1},
       pages={19\ndash 51},
}

\bib{bolsinov2004integrable}{book}{
      author={Bolsinov, Aleksei~Viktorovich},
      author={Fomenko, Anatolij~Timofeevi{\v{c}}},
       title={Integrable {H}amiltonian systems: geometry, topology,
  classification},
   publisher={CRC Press},
        date={2004},
}

\bib{bolsinov2014singularities}{article}{
      author={Bolsinov, Alexey},
      author={Izosimov, Anton},
       title={Singularities of bi-{H}amiltonian systems},
        date={2014},
     journal={Communications in Mathematical Physics},
      volume={331},
      number={2},
       pages={507\ndash 543},
}

\bib{bolsinov2006singularities}{article}{
      author={Bolsinov, Alexey~V},
      author={Oshemkov, Andrey~A},
       title={Singularities of integrable {H}amiltonian systems},
        date={2006},
     journal={Topological methods in the theory of integrable systems},
       pages={1\ndash 67},
}

\bib{bountis1984complete}{article}{
      author={Bountis, TC},
      author={Ramani, A},
      author={Grammaticos, B},
      author={Dorizzi, B},
       title={On the complete and partial integrability of non-{H}amiltonian
  systems},
        date={1984},
     journal={Physica A: Statistical Mechanics and its Applications},
      volume={128},
      number={1-2},
       pages={268\ndash 288},
}

\bib{camacho1973morse}{article}{
      author={Camacho, C},
       title={{M}orse-{S}male $\mathbb{R}^2$-actions on two-manifolds},
        date={1973},
     journal={Dynamical Systems, Editor MM Peixoto, Academic Press},
      volume={71},
       pages={74},
}

\bib{cushman2001non}{article}{
      author={Cushman, Richard},
      author={Duistermaat, JJ},
       title={Non-{H}amiltonian monodromy},
        date={2001},
     journal={Journal of Differential Equations},
      volume={172},
      number={1},
       pages={42\ndash 58},
}

\bib{derham2012differentiable}{book}{
      author={de~Rham, Georges},
       title={Differentiable manifolds: forms, currents, harmonic forms},
      series={Grundlehren der mathematischen Wissenschaften},
   publisher={Springer Berlin Heidelberg},
        date={1984},
      volume={266},
}

\bib{fedorov2006quasi}{article}{
      author={Fedorov, Yuri~N},
      author={Jovanovi{\'c}, Bo{\v{z}}idar},
       title={Quasi-{C}haplygin systems and nonholonimic rigid body dynamics},
        date={2006},
     journal={Letters in Mathematical Physics},
      volume={76},
      number={2},
       pages={215\ndash 230},
}

\bib{fomenko2014algebra}{incollection}{
      author={Fomenko, Anatoly~T},
      author={Konyaev, Andrei},
       title={Algebra and geometry through {H}amiltonian systems},
        date={2014},
   booktitle={Continuous and distributed systems},
   publisher={Springer},
       pages={3\ndash 21},
}

\bib{hempel1976three}{book}{
      author={Hempel, John},
       title={{$3$}-{M}anifolds},
   publisher={Princeton University Press, Princeton, N. J.},
        date={1976},
        note={Ann. of Math. Studies, No. 86},
}

\bib{ishida2013topological}{article}{
      author={Ishida, Hiroaki},
      author={Fukukawa, Yukiko},
      author={Masuda, Mikiya},
       title={Topological toric manifolds},
        date={2013},
     journal={Moscow Mathematical Journal},
      volume={13},
      number={1},
       pages={57\ndash 98},
}

\bib{marco2013polynomial}{article}{
      author={Marco, Jean-Pierre},
       title={Polynomial entropies and integrable {H}amiltonian systems},
        date={2013},
     journal={Regular and Chaotic Dynamics},
      volume={18},
      number={6},
       pages={623\ndash 655},
}

\bib{milnor1963morse}{book}{
      author={Milnor, John~Willard},
       title={Morse theory},
   publisher={Princeton university press},
        date={1963},
      number={51},
}

\bib{smale1960morse}{article}{
      author={Smale, Stephen},
      author={others},
       title={Morse inequalities for a dynamical system},
        date={1960},
     journal={Bulletin of the American Mathematical Society},
      volume={66},
      number={1},
       pages={43\ndash 49},
}

\bib{stolovitch2000singular}{article}{
      author={Stolovitch, Laurent},
       title={Singular complete integrability},
        date={2000},
     journal={Publications Math{\'e}matiques de l'Institut des Hautes
  {\'E}tudes Scientifiques},
      volume={91},
      number={1},
       pages={133\ndash 210},
}

\bib{vungoc2013}{article}{
      author={{V\~{u} Ng\d{o}c}, San},
      author={Wacheux, Christophe},
       title={Smooth normal forms for integrable {H}amiltonian systems near a
  focus--focus singularity},
        date={2013},
     journal={Acta Mathematica Vietnamica},
      volume={38},
      number={1},
       pages={107\ndash 122},
}

\bib{zung1996symplectic}{article}{
      author={Zung, Nguyen~Tien},
       title={Symplectic topology of integrable {H}amiltonian systems, {I}:
  {A}rnold-{L}iouville with singularities},
        date={1996},
     journal={Compositio Mathematica},
      volume={101},
      number={2},
       pages={179\ndash 215},
}

\bib{zung2014geometry}{article}{
      author={Zung, Nguyen~Tien},
      author={Van~Minh, Nguyen},
       title={Geometry of nondegenerate $\mathbb{R}^n$-actions on
  $n$-manifolds},
        date={2014},
     journal={Journal of the Mathematical Society of Japan},
      volume={66},
      number={3},
       pages={839\ndash 894},
}

\end{biblist}
\end{bibdiv}
